\documentclass[12pt,reqno]{amsart}
\usepackage[margin=1in]{geometry}
\usepackage[utf8]{inputenc}
\usepackage{setspace}
\usepackage{hyperref}
\usepackage{graphicx}
\usepackage{xcolor}
\usepackage[bitstream-charter]{mathdesign} 
\usepackage{amsmath, mathtools}
\usepackage{amsthm}
\usepackage{enumitem}
\usepackage{caption}
\usepackage{subcaption}

\usepackage{todonotes}
\usepackage{tikz}
\usepackage[normalem]{ulem}

\newtheorem{theorem}{Theorem}
\newtheorem{prop}{Proposition}

\newtheorem{definition}{Definition}
\newtheorem{lemma}{Lemma}

\newtheorem{claim}{Claim}
\newtheorem{fact}{Fact}

\hypersetup{
    colorlinks = false,
    allbordercolors = {white},
}

\tikzstyle{vert}=[shape=circle,draw=black,fill=black, inner sep=.75mm]
\tikzstyle{fixed}=[shape=rectangle,draw=black,fill=white, inner sep=1.2mm]
\tikzstyle{uncolored}=[dashed,thick]
\tikzstyle{red?}=[dashed,thick,color=red]
\tikzstyle{blue?}=[dashed,thick,color=blue]
\tikzstyle{purple?}=[dashed,thick,color=purple]
\tikzstyle{purple}=[solid,thick,color=purple]
\tikzstyle{red}=[solid,thick,color=red]
\tikzstyle{blue}=[solid,thick,color=blue]
\tikzstyle{green}=[solid,thick,color=green]
\tikzstyle{green?}=[dashed,thick,color=green]

\newcommand{\Mean}{\mathbb{E}}

\renewcommand{\S}{S}
\newcommand{\A}{A}
\newcommand{\B}{B}
\newcommand{\C}{C}
\newcommand{\CI}[1][uu'u'']{C_{#1}^{(1)}}
\newcommand{\CII}[1][uu'u'']{C_{#1}^{(2)}}
\newcommand{\CIp}[1][uu'u'']{C_{#1}^{(1)+}}
\newcommand{\CIpm}[1][uu'u'']{C_{#1}^{(1)\pm}}

\newcommand{\D}{D}
\newcommand{\E}{E}
\newcommand{\F}{F}
\newcommand{\COL}{\textsc{col}}
\newcommand{\bin}{\textrm{Bin}}

\renewcommand{\d}{\delta}

\newcommand{\rbrac}[1]{\left(#1\right)} %round brackets
\newcommand{\sbrac}[1]{\left[ #1\right]} %square brackets
\newcommand{\cbrac}[1]{\left\{ #1\right\}} %curly brackets
\newcommand{\abrac}[1]{\left| #1\right|}

\newcommand{\mc}[1]{\mathcal{#1}}
\newcommand{\tbf}[1]{\emph{#1}}

\newcommand{\eps}{\varepsilon}
\renewcommand{\k}{\kappa}

\newcommand{\nn}{\nonumber}

\def\Var{\mbox{{\bf Var}}}
\renewcommand{\P}{\mathbb{P}}

\allowdisplaybreaks

\title{The Erd\H{o}s-Gy\'arf\'as function $f(n, 4, 5) = \frac 56 n + o(n)$ --- so Gy\'arf\'as was right}
\author{Patrick Bennett}
\address{Department of Mathematics, Western Michigan University, Kalamazoo, MI, USA}
\thanks{The first author was supported in part by Simons Foundation Grant \#426894.}
\email{\tt patrick.bennett@wmich.edu}

\author{Ryan Cushman}\thanks{}
\address{Department of Mathematics,
Toronto Metropolitan University, Toronto, ON, Canada} 
\email{\tt ryan.cushman@ryerson.ca}

\author{Andrzej Dudek}
\address{Department of Mathematics, Western Michigan University, Kalamazoo, MI, USA}
\thanks{The third author was supported in part by Simons Foundation Grant \#522400.}
\email{\tt andrzej.dudek@wmich.edu}

\author{Pawe\l{} Pra\l{}at}\thanks{}
\address{Department of Mathematics,
Toronto Metropolitan University, Toronto, ON, Canada} 
\thanks{The fourth author was supported in part by NSERC Discovery Grant \#RGPIN-2022-03804.}
\email{\tt pralat@ryerson.ca}

\begin{document}

\maketitle

\begin{abstract}
A $(4, 5)$-coloring of $K_n$ is an edge-coloring of $K_n$ where every $4$-clique spans at least five colors. We show that there exist $(4, 5)$-colorings of $K_n$ using $\frac 56 n + o(n)$ colors.  This settles a disagreement between Erd\H{o}s and Gy\'arf\'as reported in their 1997 paper. Our construction uses a randomized process which we analyze using the so-called differential equation method to establish dynamic concentration. In particular, our coloring process uses random triangle removal, a process first introduced by Bollob\'as and Erd\H{o}s, and analyzed by Bohman, Frieze and Lubetzky. 

\end{abstract}

\section{Introduction}

In 1975, Erd\H{o}s and Shelah~\cite{E75} defined the following generalization of classical Ramsey numbers. 
\begin{definition}
Fix  integers $p, q$ such that $p \ge 3$ and $2 \le q \le \binom p2$. A \emph{$(p, q)$-coloring} of $K_n$ is a coloring of the edges of $K_n$ such that every $p$-clique has at least $q$ distinct colors among its edges. The Erd\H{o}s-Gy\'arf\'as function $f(n, p, q)$ is the minimum number of colors such that $K_n$ has a $(p, q)$-coloring.
\end{definition}
\noindent
We are interested in fixing $p, q$ and investigating the asymptotic behavior of $f(n, p, q)$ as $n$ tends to infinity. In particular we will be investigating $f(n, 4, 5)$. But in order to introduce the general problem, we will discuss what is known about other ``small'' pairs $(p, q)$. We start with the case where $q=2$, which is equivalent to a classical Ramsey problem. Recall that we define the Ramsey number $R_k(p)$ to be the smallest natural number $N$ such that every edge-coloring of $K_N$ using $k$ colors yields a monochromatic $p$-clique. Thus, $f(n, p, 2)$ is the smallest $k$ such that $R_k(p) > n$. The following lower bound was proved by Lefmann~\cite{L87} and the upper bound follows from the Erd\H{o}s-Szekeres ``neighborhood chasing'' argument~\cite{ES35}:
\[
2^{kp/4} \le R_k(p) \le k^{kp}.
\]
It follows for fixed $p\ge 3$ that
\[
    \Omega\rbrac{\frac{\log n}{\log \log n}} = f(n, p, 2) = O\rbrac{\log n}.
\]

Next we discuss $(3, 3)$-colorings. This case is easy but we would like to use it to preview $(4, 5)$-colorings. It is not difficult to see that $f(n, 3, 3) = \chi'(K_n)$, since a $(3,3)$-coloring is precisely a proper edge coloring of $K_n$, i.e.\ a decomposition of the edges into matchings. Later we will see that finding $f(n, 4, 5)$ also involves a type of decomposition problem (with additional constraints).
Using the well-known values of $\chi'(K_n)$ we get
\[
f(n, 3, 3) =\chi'(K_n) =\begin{cases} 
& n-1, \;\;\; \mbox{ $n$ is even},\\
& n, \;\;\; \mbox{\hskip4ex $n$ is odd.}
\end{cases}
\]

We now consider $(p, q)$-colorings where $p \ge 4$ and $q \ge 3$. It is easy to see that here we have $f(n, p, \binom p2) = \binom n2$. Thus if we examine the sequence of functions $f(n, p, 2), f(n, p, 3), \ldots,$ $f(n, p, \binom p2)$ we see it starts with at most logarithmic growth and gets larger until we see quadratic growth. Erd\H{o}s and Gy\'arf\'as~\cite{EG97} found for each $p$ the smallest value of $q$ such that $f(n, p, q)$ is at least linear in $n$ (such $q$ is called the \tbf{linear threshold}). They also found the smallest $q$ such that $f(n, p, q)$ is quadratic (the \tbf{quadratic threshold}). In particular, they showed that the linear threshold is $q = \binom p2 - p + 3$ and that the quadratic threshold is $q = \binom p2 - \lfloor p/2 \rfloor + 2$. Among several other questions posed in~\cite{EG97}, they ask the following: for fixed $p$, what is the smallest $q$ such that $f(n, p, q)$ is polynomial in $n$ (the \tbf{polynomial threshold})? They showed that  the polynomial threshold for any $p$ is at most $p$, and in particular 
\begin{equation}\label{eqn:fnpp}
f(n, p, p) \ge n^{\frac{1}{p-2}}.
\end{equation}

For $(4, 3)$-colorings, the following lower bound is due to Fox and Sudakov~\cite{FS08} and upper bound is due to Mubayi~\cite{M98}:
\[
\Omega \rbrac{ \log n} = f(n, 4, 3) \le \exp \cbrac{ O\rbrac{\sqrt{ \log n}}} = n^{o(1)}.
\]
Thus, the polynomial threshold for $p=4$ is $q=4$. 

For $(4, 4)$-colorings, the following lower bound follows from equation~\eqref{eqn:fnpp} and upper bound is due to Mubayi~\cite{M04}:
\[
n^{1/2} \le f(n, 4, 4) \le n^{1/2} \exp \cbrac{ O\rbrac{\sqrt{ \log n}}} = n^{1/2+o(1)}.
\]

Thus, we arrive at $(4, 5)$-colorings, which is  the main focus of this paper. Of course $f(n, 4, 5) = \Omega(n)$ since $q=5$ is the linear threshold for $p=4$. Moreover, Erd\H{o}s and Gy\'arf\'as \cite{EG97} paid special attention to $f(n, 4, 5)$ and gave a proof that
\[
    \frac56 (n-1) \le f(n, 4, 5) \le n,
\]
although the lower bound was previously stated by Erd\H{o}s, Elekes and F\"{u}redi \cite{E81}. Since the coefficients $5/6$ and $1$ are so close, Erd\H{o}s and Gy\'arf\'as were tempted to make a guess as to what the true coefficient should be. Erd\H{o}s thought that it should be $1$, while Gy\'arf\'as thought that it should be ``closer to $5/6$''~\cite{EG97}. Our main theorem settles this disagreement:
\begin{theorem}\label{thm:main}
We have
\[
    f(n, 4, 5) = \frac 56 n + o(n).
\]
\end{theorem}

Let us also mention that the function $f(n,p,q)$ has been extensively studied by several other researchers, see, e.g., \cite{A00, BEHK22, C19, CH18, CH20, CFLS15, FPS20, PS19, SS01, SS03}.

\subsection{Proof overview}

We outline the proof of Theorem~\ref{thm:main}. The lower bound was proved by Erd\H{o}s and Gy\'arf\'as~\cite{EG97}, and for the sake of completeness we will restate their proof in Section \ref{sec:process}. For the upper bound, it clearly suffices to show that for any fixed $\eps > 0$ we have  $f(n, 4, 5) \le \frac 56 n + \eps n$ for all sufficiently large $n$. We show that there exists some randomized coloring procedure using $\frac56 n + \eps n$ colors such that the probability of getting a $(4, 5)$-coloring is positive for sufficiently large $n$. We will then define a procedure using two \emph{phases}. The \emph{first phase} will (if successful) use $\frac 56 n + \frac12 \eps n$ colors to color almost all the edges of $K_n$ using a randomized coloring process, and the analysis of this phase will be the main work of this paper. The \emph{second phase} will color the remaining uncolored edges using a much simpler random coloring and a fresh set of $\frac12 \eps n$ colors. Our analysis of the first phase of the process will show that with positive probability it outputs a partial coloring with nice properties that will allow us to easily show that the second phase successfully finishes a $(4, 5)$-coloring with positive probability, which completes the proof.

For the first phase we will use the differential equation method (see~\cite{BD20} for a gentle introduction) to establish dynamic concentration of our random variables. The origin of the differential equation method stems from work done at least as early as 1970 (see Kurtz \cite{Kurtz1970}), and which was developed into a very general tool by Wormald \cite{W1995, W1999} in the 1990's. Indeed, Wormald proved a ``black box'' theorem, which gives dynamic concentration so long as some relatively simple conditions hold. Warnke \cite{Warnke2020} recently gave a short proof of a somewhat stronger black box theorem. For our purposes the existing black box theorems are insufficient, but we are still able to analyze our process using fairly standard arguments that resemble previous analyses of other processes.

The analysis of the second phase will be based on the Lov\'asz Local Lemma.

\subsection{Tools}\label{sec:tools}

We will be using the following forms of Chernoff's bound (see, e.g., \cite{JLR}). 

\begin{lemma}[Chernoff bound]
Let $X\sim \bin(n,p)$ and $\mu = \E(X) = np$. Then, for all $0<\delta<1$
\begin{equation}\label{Chernoff_upper}
\Pr(X \ge (1+\delta) \mu) \le \exp(-\mu \delta^2/3)
\end{equation}
and
\begin{equation}\label{Chernoff_lower}
\Pr(X \le (1-\delta)\mu) \le \exp(-\mu \delta^2/2). 
\end{equation}
%Also for any $t\ge 0$, we have
%\begin{equation}\label{Chernoff:upper2}
%\Pr(X\ge \mu+t) \le \exp\left( -\frac{t^2}{2(\mu+t/3)}\right).
%\end{equation}
\end{lemma}

We will also need Freedman's inequality~\cite{F1975}, which we state next.

\begin{lemma}[Freedman's inequality]\label{lem:Freedman}
Let $W(i)$ be a supermartingale with $\Delta W(i) \leq D$ for all $i$, and let \newline $V(i) :=\displaystyle \sum_{k \le i} \Var[ \Delta W(k)| \mc{F}_{k}]$.  Then,
\[
\P\left[\exists i: V(i) \le b, W(i) - W(0) \geq \lambda \right] \leq \displaystyle \exp\left(-\frac{\lambda^2}{2(b+D\lambda) }\right).
\] 
\end{lemma}

Finally, let us state the Lov\'asz Local Lemma (LLL)~\cite{AS}. For a set of events $\mc{A}$ and a graph $G$ on vertex set $\mc{A}$, we say that $G$ is a  \tbf{dependency graph} for $\mc{A}$ if each event $A \in \mc{A}$ is not adjacent to events which are mutually independent. 

\begin{lemma}[Lov\'asz Local Lemma] \label{lem:LLL}
Let $\mc{A}$ be a finite set of events in a probability space $\Omega$ and let $G$ be a dependency graph for $\mc{A}$. Suppose there is an assignment $x:\mc{A} \rightarrow [0, 1)$ of real numbers to $\mc{A}$ such that for all $A \in \mc{A}$ we have
\begin{equation}\label{eqn:LLLcond}
    \Pr(A) \le x(A) \prod_{B \in N(A)} (1-x(B)).
\end{equation}
Then, the probability that none of the events in $\mc{A}$ happen is 
\[
    \Pr\rbrac{\bigcap_{A \in \mc{A}} \overline{A}} \ge \prod_{A \in \mc{A}} (1-x(A)) >0.
\]
\end{lemma}

\subsection{Organization of the paper}
In Section~\ref{sec:process} we motivate and formally define the process for the first phase of our coloring procedure.  Our analysis of the process will hinge on our ability to maintain good estimates of a family of random variables which change with each step of the process. In Section~\ref{sec:variables} we define our family of variables, and in Section~\ref{sec:goodevent} we state the bounds that we intend to prove for our random variables. In Sections~\ref{sec:QY}--\ref{sec:crude} we bound the probability that any of our random variables violate the stated bounds until almost all the edges are colored. This will finish the first phase of the proof, which leaves just a few uncolored edges. Finally, in Section~\ref{sec:finishing} we show how to color such uncolored edges.

%%%%%%%%%%%%%%%%% PROCESS %%%%%%%%%%%%%%%%%%%%%%%%%%%%
\section{The coloring process}\label{sec:process}

First we give some motivation. Let us start by seeing the proof of the lower bound for Theorem~\ref{thm:main}, which was given by Erd\H{o}s and Gy\'arf\'as~\cite{EG97}. The proof will illuminate what needs to be done to achieve an asymptotically matching upper bound (and we will comment on that after the proof).

\begin{theorem}[\cite{EG97}]
We have
\[
  f(n, 4, 5) \ge \frac56 (n-1).
\]
\end{theorem}
\begin{proof}
Suppose we have a $(4, 5)$-coloring of a graph of order $n$ using a set of colors $C$. For each $c \in C$ let $G_c$ be the graph of order $n$ with only the $c$-colored edges. Note that $G_c$ can never have a connected component with more than two edges (i.e.\ all components have at most three vertices and there are no monochromatic triangles). Thus every component of $G_c$ is either $P^0$, $P^1$, or $P^2$, where $P^j$ denotes a path on $j$ edges. For $0 \le j \le 2$ let $x_j$ be the total number of components $P^j$ in all the graphs $G_c$, $c \in C$. Thus 
\begin{equation}\label{eqn:EGproof1}
   x_0 + 2x_1 + 3x_2 = n|C| 
\end{equation}
and 
\begin{equation}\label{eqn:EGproof2}
    x_1 + 2x_2 = \binom n2.
\end{equation}
Note also that whenever we have a component in color $c$ with two edges on three vertices (i.e.\ a component counted by $x_2$), the third edge in that triangle must be a component counted by $x_1$. Thus, $x_1 \ge x_2$,
and hence, $x_0 + \frac13 (x_1 - x_2) \ge 0$.
But then, using~\eqref{eqn:EGproof1} and \eqref{eqn:EGproof2}, we get
\begin{align*}
    |C|  = \frac{x_0 + 2x_1 + 3x_2}{n}
         \ge \frac{x_0 + 2x_1 + 3x_2 -  \rbrac{x_0 + \frac13 (x_1 - x_2)}}{n}
         = \frac{\frac53 (x_1 + 2x_2)}{n}  = \frac56 (n-1).
\end{align*}
\end{proof}

From the proof we can see that the only way to achieve equality would be if $x_0=0$ and $x_1=x_2.$ Erd\H{o}s~\cite{EG97} expressed doubt that any coloring could come close to that. Indeed, he suspected that if we have $x_0=o(n^2)$ then we must also have $x_2=o(n^2)$, i.e.\ $x_1$ dominates everything and essentially all the graphs $G_c$ are matchings with $n/2 -o(n)$ edges. Such a coloring would have $|C| = n-o(n)$. Indeed, Erd\H{o}s and Gy\'arf\'as~\cite{EG97} gave such a coloring to prove the upper bound $f(n, 4, 5) \le n.$

To prove Theorem~\ref{thm:main}, we will need to get a coloring with $x_0 = o(n^2)$ and $x_1=x_2 +o(n^2)$. In other words, for almost every $P^1$ component  in some graph $G_c$, its two endpoints are also the ends of some $P^2$ component in some $G_{c'}$. Thus we are motivated to consider a process which at each step $i$ colors the edges of some triangle $T_i$ (whose edges have no colors yet), giving two of them the same color and the third one a different color. The intent is to create one $P^1$ component in a color $c$ and a $P^2$ component in another color $c'$. When a vertex $v$ is incident to an edge of some color $c$ we say $v$ has been \tbf{hit} by $c$. To ensure that our components do not accidentally become larger than intended, at each step we will have to make sure to choose $c, c'$ that have not already hit the vertices they are about to hit. 

There are many ways we could choose the triangle $T_i$ whose edges we will color at step $i$. We will use what seems to be the most natural (and well-studied) candidate: the \tbf{random triangle removal process} first introduced by Bollob\'as and Erd\H{o}s (see~\cite{B98, B00}).  In this process one starts with $G_R(0)=K_n$ and at each step~$i$ removes the edges of one triangle chosen uniformly at random from all triangles in $G_R(i)$, stopping only when the graph becomes triangle-free.  Bollob\'as and Erd\H{o}s conjectured that the number of edges remaining  at the end of this process (i.e.\ edges not in the triangle packing) is $\Theta(n^{3/2})$ a.a.s.\ (\emph{asymptotically almost surely}, that is, with probability tending to one as $n \to \infty$). The best known estimate (both upper and lower bounds) on the number of edges remaining is $n^{3/2 + o(1)}$ by Bohman, Frieze and Lubetzky~\cite{BFL15}. We will not need the full power of their result, but for our convenience we will use a few facts they proved in their analysis of the process.

For our coloring process, at each step $i$ we will choose our triangle $T_i$ uniformly at random from all triangles whose edges are all uncolored. We will then randomly choose an \tbf{orientation} for $T_i$, meaning that we choose which of the three edges will be in a $P^1$ component (meaning the other two will make a $P^2$). We will then randomly choose two ``suitable'' colors $c_i, c'_i$ to assign to the edges of $T_i$. In the end we will use a somewhat complicated rule to determine which colors are ``suitable'' here. Of course our rule must not violate the constraint for $(4, 5)$-coloring, which requires for each set of four vertices to have five different colors among its six edges. So somehow our process must prevent the creation of any set of four vertices having two repeated colors (or one color repeated twice).

Suppose $T_i = \{u, u', u''\}$ is the selected triangle. We will choose the orientation such that $u'u''$ will be assigned the color $c_i'$ and the other two edges will be assigned $c_i$. In this case, we say that the triangle is oriented {\em away from $u$}. Obviously, by our previous discussion we should choose the colors such that $c_i'$ has not hit $u'$ or $u''$ and $c_i$ has not hit $u, u'$, or $u''$. Thus throughout the process  our coloring will have no color components with more than two edges and our color components $P^1$ and $P^2$ will come in pairs sharing endpoints. This requirement already avoids many of the ways our coloring could violate the constraint for a $(4,5)$-coloring. For example, since our color components have at most two edges, we cannot have four vertices containing three edges of the same color. Thus any violation of a $(4,5)$-coloring must involve two different colors, each appearing twice in the same set of four vertices. The rule we have already described also avoids the two configurations illustrated in Figure~\ref{fig1}. 

\noindent
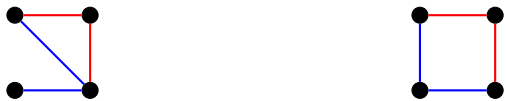
\begin{figure}[h!]
\begin{center}
\begin{tikzpicture}[scale=.5]
	\node (ll) at (0,0) [vert] {};
	\node (lr) at (2,0) [vert] {};
	\node (ur) at (2,2) [vert] {};
	\node (ul) at (0,2) [vert] {};
	\draw [blue] (ll) -- (lr);
	\draw [red] (ur) -- (ul);
	\draw [red] (ur) -- (lr);
	\draw [blue] (ul) -- (lr);
\end{tikzpicture}
\qquad \qquad \qquad \qquad \qquad
\begin{tikzpicture}[scale=.5]
	\node (ll) at (0,0) [vert] {};
	\node (lr) at (2,0) [vert] {};
	\node (ur) at (2,2) [vert] {};
	\node (ul) at (0,2) [vert] {};
	\draw [blue] (ll) -- (lr);
	\draw [blue] (ll) -- (ul);
	\draw [red] (ur) -- (ul);
	\draw [red] (ur) -- (lr);
\end{tikzpicture}
\end{center}
    \caption{Two possible configurations that would violate a $(4,5)$-coloring.}
    \label{fig1}
\end{figure}

However, unless we impose some additional rules for choosing our colors, our process would allow the configurations depicted in Figure~\ref{fig2} that would violate the constraint for a $(4, 5)$-coloring.

% \noindent
% \begin{center}
% \begin{tikzpicture}[scale=1]
% 	\node (ll) at (0,0) [vert] {};
% 	\node (lr) at (2,0) [vert] {};
% 	\node (ur) at (2,2) [vert] {};
% 	\node (ul) at (0,2) [vert] {};
% 	\draw [blue] (ll) -- (lr);
% 	\draw [blue] (ul) -- (ur);
% 	\draw [red] (ul) -- (ll);
% 	\draw [red] (ur) -- (lr);
% \end{tikzpicture}
% \qquad \qquad \qquad
% \begin{tikzpicture}[scale=1]
% 	\node (ll) at (0,0) [vert] {};
% 	\node (lr) at (2,0) [vert] {};
% 	\node (ur) at (2,2) [vert] {};
% 	\node (ul) at (0,2) [vert] {};
% 	\draw [blue] (ul) -- (ur);
% 	\draw [blue] (lr) -- (ul);
% 	\draw [red] (ur) -- (lr);
% 	\draw [red] (ul) -- (ll);
% \end{tikzpicture}
% \end{center}

\begin{figure}[h!]
    \centering
    \begin{subfigure}[b]{0.3\textwidth}
    \centering
     \begin{tikzpicture}[scale=.5]
	\node (ll) at (0,0) [vert] {};
	\node (lr) at (2,0) [vert] {};
	\node (ur) at (2,2) [vert] {};
	\node (ul) at (0,2) [vert] {};
	\draw [blue] (ll) -- (lr);
	\draw [blue] (ul) -- (ur);
	\draw [red] (ul) -- (ll);
	\draw [red] (ur) -- (lr);
\end{tikzpicture}
    \caption{}
    \label{fig:4cycle}
    \end{subfigure}
    \begin{subfigure}[b]{0.3\textwidth}
    \centering
     \begin{tikzpicture}[scale=.5]
	\node (ll) at (0,0) [vert] {};
	\node (lr) at (2,0) [vert] {};
	\node (ur) at (2,2) [vert] {};
	\node (ul) at (0,2) [vert] {};
	\draw [blue] (ul) -- (ur);
	\draw [blue] (lr) -- (ul);
	\draw [red] (ur) -- (lr);
	\draw [red] (ul) -- (ll);
\end{tikzpicture}
    \caption{}
    \label{fig:violation}
    \end{subfigure}
    \caption{Two additional configurations that would violate a $(4,5)$-coloring and require extra consideration.}
    \label{fig2}
\end{figure}
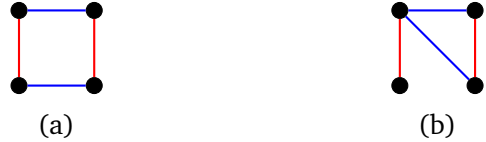

We will have to impose two more rules to avoid these violations. Of course, we are still trying to use only $\frac 56 n + o(n)$ colors. We pause to comment that based on the rules we have stated so far, it is heuristically plausible that we could use the process to color almost all edges using $\frac 56 n + o(n)$ colors. Indeed, after $i \le \frac 13 \binom n2$ steps there are $i$ colored triangles, and so each vertex $v$ should be incident with about $3i/n$ of them.  About $1/3$ of those triangles will get oriented in a way that means $v$ gets hit by only one color whereas $2/3$ of these triangles have $v$ getting hit by two colors. Thus the number of colors that have hit $v$ should be about 
\[
\frac {3i}{n} \cdot \sbrac{\frac13 \cdot 1 + \frac 23 \cdot 2} = \frac {5i}{n} \le \frac {5\cdot \frac13 \binom n2 }{n} \le \frac 56 n
\]
and so, heuristically, given any vertex it should be possible to choose a color that has not hit it yet (until almost all edges are colored). However, running this simpler process might require substantial extra colors to make it into a $(4,5)$-coloring afterwards. Thus, we impose the following additional rules into our process.

We will avoid the configuration in Figure \ref{fig:4cycle} (an \tbf{alternating 4-cycle}) by ``brute force'' adjustment: when we choose our colors we will simply refuse to create such a cycle (i.e.\ color choices that would create one are eliminated from consideration and we randomly choose from the remaining colors). While this rule does make the process more challenging to analyze, we will see that it does not reduce the number of choices we have for colors too significantly.

To avoid the configuration in Figure \ref{fig:violation}, it is tempting to say that we will use ``brute force" again and simply refuse to make it. However, some thought reveals that this idea is not too promising if we want to use only $\frac 56 n + o(n)$ colors. Indeed, when we have colored, say, about half of the edges, a vertex $v$ should be in some linear number of triangles oriented away from $v$. Unless our process has a rule to prevent it, we would expect to see some linear number of colors (like the colors $c, \ldots, c'$ in Figure~\ref{fig3}) appearing across from~$v$ in those triangles. None of those colors can be allowed to hit $v$, since then we would get Figure \ref{fig:violation}. However, the simpler process (without additional rules) was using very close to $\frac 56 n$ colors and to come close to $\frac 56 n$ colors we need to make sure almost every vertex gets hit by almost every color. Thus, this proposed ``brute force'' rule is not viable. 

\noindent
\begin{figure}[!h]
\begin{center}
\begin{tikzpicture}[scale=.5]
	\node (lr) at (2,0) [vert] {};
	\node (ur) at (2,2) [vert] {};
	\node (y) at (-2,0) [vert] {};
	\node (x) at (-2,2) [vert] {};
	\node [label=$\ldots$] (dots) at (0,2)  {};
	\node (v) at (0,.5) [vert, label=below:$v$] {};
	\draw [blue] (v) -- (ur);
	\draw [blue] (lr) -- (v);
	\draw [red] (ur) -- node[label=right:{$c'$}]{}(lr);
	\draw [cyan] (v) --(x);
	\draw [cyan] (v) --(y);
	\draw [orange] (y) -- node[label=left:{$c$}]{} (x);
\end{tikzpicture}
\end{center}
\caption{If using a ``brute force'' adjustment to the process, there would be a linear number of colors $c, \ldots, c'$ across from $v$.}
    \label{fig3}
\end{figure}
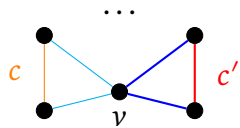

In order to overcome this issue, we will do the following. Each vertex $v$ will have some small 
linear set of \tbf{special colors} $S_v$, which will be the only colors we allow to appear opposite from $v$ in triangles oriented away from $v$. To avoid the configuration in Figure \ref{fig:violation}, we will make sure that $v$ is never hit by any color in $S_v$. 

\subsection{First phase}

In this subsection we define the random process we will use to color almost all edges. 

Suppose we have a set $\overline{\COL}$ of $\frac 56 n + \eps n$ colors for some $0<\eps<1/100$. 
Our coloring method has two phases and the first phase will need almost all the colors. We let $\COL \subseteq \overline{\COL}$ be some subset of $\frac 56 n + \frac{\eps}{2} n$ colors. We will use colors in $\COL$ for the first phase and reserve the rest for the second phase. We will now start describing the first phase in detail, which we motivated before this subsection. 

First, independently for each vertex $v$ and color $k$, we put $k$ into the set $S_v$ with probability 
\[
s:=\frac{ \frac{\eps}{2} }{\frac56 + \frac{\eps}{2}  }.
\]
The colors in $S_v$ will not be allowed to hit $v$, and they will be the only colors allowed to appear across from $v$ in triangles oriented away from $v$. Note that $s$ was chosen so that the number of colors we allow to hit $v$, i.e.\ $|\COL \setminus S_v|$, has expectation $\frac 56 n$.

We will need the following definitions. 

\begin{definition}
An \tbf{alternating $(uv, k)$-path} is a $u-x-y-v$ path such that edges $ux$ and $vy$ are colored the same color and edge $xy$ is colored $k$. 
\end{definition}

\begin{definition}
Let $k$ and $k'$ be colors in $\COL$.
\begin{itemize}
\item We say $k$ is \tbf{available at a vertex} $u$ at step $i$ if 
 $k \notin \S_u$, and
 $u$ has not been hit by $k$.

\item We say $k$ is \tbf{available at an edge} $uv$ at step $i$ if 
$uv$ is uncolored,
$k$ is available at each of the vertices $u$ and $v$, and
there is no alternating $(uv, k)$-path.

\item We say $k'$ is \tbf{1-available at a triple} $(u, u', u'')$ at step $i$ if 
$k' \in \S_u$ and
$k'$ is available at the edge $u'u''$.

\item We say $k$ is \tbf{2-available at a triple} $(u, u', u'')$ at step $i$ if  $k$ is available at the edges $uu'$ and $uu''$.
We say a pair $(k, k')$ is \tbf{available at a triple} $(u, u', u'')$ at step $i$ if 
$k'$ is 1-available at $(u,u',u'')$ and
$k$ is 2-available at $(u,u',u'')$.
\end{itemize}
(Note that this definition implies all edges in $uu'u''$ are uncolored. Also, the roles of $u'$ and $u''$ are interchangeable but the role of $u$ is different). 
\end{definition}

Now, we are ready to define the process.

\vspace{2ex}

\begin{enumerate}[label={\bf Substep \arabic*.\!}, wide, labelindent=0pt]
\noindent\begin{minipage}{.6\textwidth}
\item (Initialization) Start with a complete, uncolored graph $K_n=(V,E)$ on $n$ vertices. 
Recall that for each $u \in V$, we have a set $\S_u$ of colors from $\COL$. Colors from $S_u$ can be assigned as the opposite color to $u$ when the triangle  is oriented away from $u$ (see opposite figure). On the other hand, these colors are not not allowed to touch~$u$.
\end{minipage}
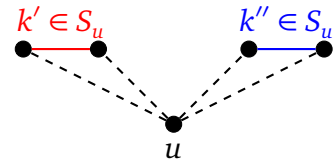
\begin{minipage}{.4\textwidth}
\begin{center}
\begin{tikzpicture}
	\node (u) at (0,0) [vert,label=below:$u$] {};
	\node (v) at (1,1) [vert] {};
	\node (u'') at (2,1) [vert] {};
	\node (x) at (-1,1) [vert] {};
	\node (y) at (-2,1) [vert] {};
	\draw [uncolored] (u) -- (v);
	\draw [uncolored] (u) -- (u'');
	\draw [uncolored] (u) -- (x);
	\draw [uncolored] (u) -- (y);
	\draw [blue] (v) -- node[above] {$k''\in S_u$} (u'');
	\draw [red,] (x) -- node[above] {$k'\in S_u$} (y);
\end{tikzpicture}
\end{center}
\end{minipage}

\vspace{2ex}

\noindent\begin{minipage}{.6\textwidth}
\item (Triangle) At step $i$, choose an uncolored triangle $T_i$ uniformly at random from the set of uncolored triangles.
\end{minipage}
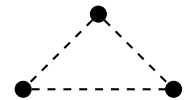
\begin{minipage}{.4\textwidth}
\begin{center}
\begin{tikzpicture}
	\node (x) at (0,0) [vert] {};
	\node (y) at (2,0) [vert] {};
	\node (v) at (1,1) [vert] {};
	\draw [uncolored] (v) -- (x);
	\draw [uncolored] (v) -- (y);
	\draw [uncolored] (x) -- (y);
\end{tikzpicture}
\end{center}
\end{minipage}

\vspace{2ex}

\noindent\begin{minipage}{.6\textwidth}
\item (Orientation) Choose one vertex $u$ uniformly at random from the three vertices of the selected triangle $T_i$. The triangle will be oriented away from $u$. Label the other vertices $u'$ and $u''$.
\end{minipage}
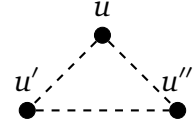
\begin{minipage}{.4\textwidth}
\begin{center}
\begin{tikzpicture}
	\node (u') at (0,0) [vert,label=above:$u'$] {};
	\node (u'') at (2,0) [vert,label=above:$u''$] {};
	\node (u) at (1,1) [vert,label=above:$u$] {};
	\draw [uncolored] (u) -- (u');
	\draw [uncolored] (u) -- (u'');
	\draw [uncolored] (x) -- (y);
\end{tikzpicture}
\end{center}
\end{minipage}

\vspace{2ex}

\noindent\begin{minipage}{.6\textwidth}
\item (Color the triangle) Choose a pair of colors $(k, k')$ uniformly at random from all pairs that are available at triple $(u, u', u'')$ (or terminate if there is no such pair). Note that we can choose $k$ and $k'$ independently from each other. More specifically, we choose $k'$ uniformly at random from all colors such that $k'\in \S_u$ and $k'$ is available at $u'u''$. Independent from the choice of $k'$ we choose $k \notin S_u$ uniformly at random from all colors such that $k$ is available at both $uu'$ and $uu''$. Color $uu'$ and $uu''$ with $k$ and $u'u''$ with $k'$.

\end{minipage}
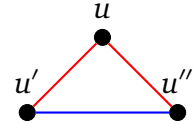
\begin{minipage}{.4\textwidth}
\begin{center}
\begin{tikzpicture}
	\node (u') at (0,0) [vert,label=above:$u'$] {};
	\node (u'') at (2,0) [vert,label=above:$u''$] {};
	\node (u) at (1,1) [vert,label=above:$u$] {};
	\draw [red] (u) -- (u');
	\draw [red] (u) -- (u'');
	\draw [blue] (u') -- (u'');
\end{tikzpicture}
\end{center}
\end{minipage}

\vspace{2ex}

\item If there are more uncolored triangles, then go back to Substep 2 and carry out step $i+1$. Otherwise, terminate.

\end{enumerate}

Note that there are two possible endings of the process: it could finish at Substep 4 because no pair of colors is available, or it could finish at Substep 5 because no uncolored triangles remain. Bohman, Frieze and Lubetzky \cite{BFL15} showed that the random triangle removal process a.a.s.\ does not terminate until $n^{3/2 + o(1)}$ edges remain, and so our process a.a.s.\ does not terminate at Substep 5 until the number of uncolored edges is $n^{3/2 + o(1)}$. Most of this paper is devoted to showing that a.a.s.\ our coloring process does not terminate at Substep 4 until we have colored almost all the edges. When the process terminates, we will we move to the second phase of our coloring procedure, which will assign colors to the remaining uncolored edges.

We note that some similar ideas were used by Guo, Patton and Warnke \cite{GPW20}. In particular they used a coloring process assigning colors one at a time where each color was chosen uniformly at random from all ``available" colors (for some appropriate definition of ``available"). 

\subsection{Second phase}
The second phase will use the set of $\frac{\eps}{2} n$ \tbf{reserved colors} $\overline{\COL} \setminus \COL$. Each edge that still needs to be colored will get one of the reserved colors chosen uniformly at random. 
Our analysis of the first phase will show that it produces a partial coloring that enjoys several useful ``pseudorandom'' properties (i.e.\ properties that one would expect to see in a simpler random coloring where each edge has an independent random color). These properties will allow us to argue that the remaining edges are relatively easy to color. 
We will use the Lov\'asz Local Lemma to show that with positive probability the resulting coloring is a $(4, 5)$-coloring and so, by the trivial probabilistic method, there exists an appropriate extension of the partial coloring we produced in the first phase to a complete $(4, 5)$-coloring.

\section{System of random variables}\label{sec:variables}
Our analysis of the process in the first phase will proceed by the differential equation method. As usual, we will define a family of random variables which we will \tbf{track} throughout the process, meaning that we will obtain asymptotically tight estimates which hold a.a.s.. For each tracked variable there will be a deterministic function, called the \tbf{trajectory}, such that a.a.s.\ the tracked variable is asymptotically equal to its trajectory. Our family of variables will also include some for which we prove only crude upper bounds (but which we do not track). 

A beautiful aspect of the differential equation method is that often the trajectories of random variables can be guessed using the right intuition and heuristics. 
Fortunately we will see that this is the case for our process. Indeed, our family of random variables will have elementary trajectories which we can guess using heuristics. In the next subsection we describe these heuristics, and in the following subsection we define our random variables. As we define each variable we state its trajectory.  
 
 \subsection{Heuristics}\label{sec:heuristics}

Before we start listing the random variables, let us go over the heuristic assumptions. We define the ``scaled time'' parameter 
\begin{equation}\label{eqn:tdef}
    t = t(i) := i/n^2.
\end{equation}
At each step $i$ we color three edges, so the total number of colored edges at that step is $3i = 3n^2 t$. Heuristically, the probability that an edge is colored is \[
\frac{3n^2 t}{\binom n2} \approx 6t.
\]
Thus, in particular we predict that in many ways the uncolored graph should resemble $G(n, p)$ with 
\begin{equation}\label{eqn:pdef}
  p=p(t) := 1 - 6t.  
\end{equation}

We would also like a heuristic for the probability that some vertex $u$ has been hit by a color $k \notin \S_u$. In this process, $u$ should be getting hit by colors at about the same rate throughout the process. In fact, the proportion of colors in $\COL \setminus S_u$ that have hit $u$ should be about the same as the proportion $1-p$ of edges in the graph we have colored. Thus, we heuristically assume that the probability $k \notin S_u$ has hit $u$ is $1-p$.

Recall that 
\[
s=  \frac{ \frac{\eps}{2} }{\frac56 + \frac{\eps}{2}  }
\]
is the probability that (for some fixed color $k$ and vertex $u$) $k$ is chosen to be in $\S_u$. Note that the expected number of colors in $|\COL \setminus S_u|$ is $(1-s)|\COL| = \frac 56 n$ and so
\[
|\COL|= \frac 56 n \cdot \frac{1}{1-s}.
\]

We will need to pay careful attention to alternating paths to analyze our process. Heuristically, for some uncolored edge $e$ and a color $k$, we will assume that there is some function $r(t)$ which we treat as the probability there is no $(uv, k)$-alternating path at time $t$. We will guess the appropriate function $r(t)$ using a Poisson heuristic. For a Poisson random variable $X$, if $\lambda = \E[X]$ then $\P(X=0) = e^{-\lambda}$. 
If we let $X$ be the number of $(uv, k)$-alternating path at time $t$, then we ought to have
\[
\E[X] \approx n^2 \cdot (1-p)^3 \cdot \rbrac{\frac{1}{|\COL|}}^2,
\]
since we have about $n^2$ choices for possible vertices $x$ and $y$ in $u-x-y-v$, each of the edges $ux$, $xy$ and $yv$ are colored with probability $1-p$, $xy$ has the color $k$ with probability $1/|\COL|$, and the edges $ux$ and $yv$ have the same color with probability $1/|\COL|$ as well.
Now substituting the value of $|\COL|$ gives
\[
\E[X] \approx n^2 \cdot (1-p)^3 \cdot \rbrac{\frac{1}{\frac 56 n \cdot \frac{1}{1-s}}}^2
= \frac{36}{25} (1-s)^2 (1-p)^3 = \frac{7776}{25} (1-s)^2 t^3.
\]
Consequently, we heuristically guess that 
\[
r(t) = \exp \cbrac{- \frac{7776}{25} (1-s)^2 t^3}.
\]
Note that for all $t \le 1/6$ we have
\[
r(t) \ge r(1/6) = \exp \cbrac{ - \frac{36}{25}(1-s)^2} \ge \exp \cbrac{ - \frac{36}{25}} > \frac{1}{5}.
\]

\subsection{Variables}

In this subsection, we introduce our family of variables. We start with the variables we intend to track, meaning that we will show that a.a.s.\ each of these variables stays within a relatively small interval centered around its trajectory. Formally we will use many random variables that are actually sets (not numbers), and when we say we ``track" them we mean that we track their cardinalities. We will often abuse notation and omit absolute value signs for the cardinality of sets, i.e.\ we write $S$ to denote either the set $S$ or its cardinality. In context there should be no confusion. At the end of this subsection we will define a few more variables for which we will obtain only crude upper bounds. 

Roughly speaking, the differential equations method is a way to formally argue that a.a.s.\ certain conditions (bounds on random variables) are maintained as the process runs. Often the goal is to argue that the process does not fail until almost all edges are colored. Thus, our choice of random variables will be motivated by what the process needs to keep going. In our case, the process needs two things: first it needs to be able to choose an uncolored triangle (i.e.\ the process does not terminate at Substep~4), and then it needs to have some choice of colors for that triangle that obey our coloring rules (i.e.\ the process does not terminate at Substep~5). Thus, our family of random variables will include one counting the number of uncolored triangles (see the variable $Q$ below), as well as ones counting the number of choices for colors we have for each such triangle (see variables $C^{(1)}, C^{(2)}$).  For the differential equation method to work we will need a ``closed system'' of variables, meaning that if we condition on the current state of the process then the expected one-step change of any variable in our family can be (approximately) written in terms of variables in our family. Thus, our family will have to include several other variables. 

We start with the variables used by Bohman, Frieze and Lubetzky~\cite{BFL10} for the triangle removal process. This includes $Q$ which is clearly important, as well as another kind of variable which is necessary to make a closed system with $Q$. 
\begin{definition}
Let $Q=Q(i)$ be the set of triangles where all three edges are uncolored at step $i$. For each $u, u'$ we let $Y_{uu'}=Y_{uu'}(i) $ be the set of vertices $u''$ such that both $uu''$ and $u'u''$ are uncolored.
\end{definition}
%Note that although the authors of~\cite{BFL15} provided a much more detailed analysis of the triangle removal process (using more variables), we will not use it since the simpler family of variables from~\cite{BFL10} suffices for our purposes. 
Recalling \eqref{eqn:tdef} and \eqref{eqn:pdef}, the natural heuristic guess for the trajectories (also proved formally in~\cite{BFL10}) is
\[
Q(i) \approx \binom n3 p^3 \approx \frac 16 n^3 p^3 = n^3q(t)  \qquad \text{and}\qquad Y_{uu'}(i) \approx np^2 =ny(t),
\]
where
\[
q(t) := \frac16 p^3 \qquad\text{and}\qquad y(t) := p^2.
\]
We will call these functions $q(t), y(t)$ (i.e.\ trajectories with the power of $n$ removed) \tbf{scaled trajectories}. Before moving to the variables that count color choices, we briefly explain how $Q$ and the $Y$ variables form a closed system. 

Let $(\mc{F}_i)_{i \ge 0}$ be the ``natural filtration" of the process. In particular, conditioning on $\mc{F}_i$ tells us exactly what our partial coloring looks like at step $i$. More formally, our probability space consists of all possible maximal sequences of steps (specifying at each step which triangle, orientation, and colors are chosen), and the partition $\mc{F}_i$ groups these sequences according to what happens on the first $i$ steps. The work of Bohman, Frieze and Lubetzky \cite{BFL10} implies that 
\[
\E[\Delta Q(i) | \mc{F}_i] = -  \sum_{uu' \in E(i)} \frac{Y_{uu'}(i)^2}{Q(i)}  +O(1)
\]
(where $E(i)$ is the set of uncolored edges at step $i$) and
\[
\E[\Delta Y_{uu'}(i) | \mc{F}_i] = -  \sum_{u'' \in Y_{uu'}(i)} \frac{Y_{uu''}(i) + Y_{u'u''}(i) +O(1)}{Q(i)}.
\]
Since the conditional expected one-step change of $Q$ and any of the $Y$ variables can be approximately written using only the variables $Q$ and $Y$, we have a closed system. However, for our coloring process we need several more variables that count color choices. We will now extend our family to include not only variables of types $\CI, \CII$ (which  count choices of colors given a fixed oriented triangle), but also several more variables needed to make the system closed again. We will verify later that this system is indeed closed. 

The variables of types $A$ through $F$ will all count triples $(u, u', u'')$ and pairs $(k, k')$ that are available at $(u, u', u'')$. For each of these variables we fix some set of vertices and/or colors and count extensions of the fixed set. 
To illustrate the substructures that these variables count, we will include diagrams that use the following conventions. Closed circles represent vertices that vary (based on some constraint), and open squares represent fixed vertices. Dashed, colored edges represent uncolored edges that have that color available at that edge; a dashed, black edge is a general uncolored edge; and a solid, colored edge is an edge with that color.  For example, Figure~\ref{fig4} would indicate that we are fixing $u, u', u''$ and counting pairs $k, k'$ such that $k$ is available at $uu'$ and $uu''$ and $k'$ is available at $u'u''$. 
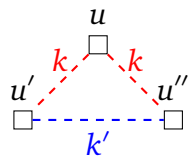
\begin{figure}[!h]
\begin{center}
\begin{tikzpicture}
	\node (u') at (0,0) [fixed,label=above:$u'$] {};
	\node (u) at (1,1) [fixed,label=above:$u$] {};
	\node (u'') at (2,0) [fixed,label=above:$u''$] {};
	\draw [red?] (u) -- node [above] {$k$} (u');
	\draw [blue?] (u') -- node [below] {$k'$} (u'');
	\draw [red?] (u) -- node [above] {$k$} (u'');
\end{tikzpicture}
\end{center}
\caption{A demonstration of the diagram conventions used in this section.}
\label{fig4}
\end{figure}
First we define the type $A$ variables, where $u'$, $u''$ and $k'$ are fixed. 

\begin{definition}
For each edge $u'u''$ and each color $k' \notin S_{u'} \cup S_{u''}$ we define the random variable $\A_{u'u'',k'}=\A_{u'u'',k'}(i)$ to be the set of pairs $(u, k)$ such that $k$ is available at $uu'$ and $uu''$,  and $k' \in S_u$.
\end{definition}
Note that technically our definition above does not assume that $k'$ is available at $u'u''$. However whenever that happens to be the case we have for all $(u, k) \in \A_{u'u'',k'}$ that the color pair $(k, k')$ is available at the oriented triangle $(u, u', u'')$.
\begin{figure}[h!]
\begin{center}
\begin{tikzpicture}
	\node (u') at (0,0) [fixed,label=above:$u'$] {};
	\node (u) at (1,1) [vert,label=above:$u$ s.t. $k' \in S_u$] {};
	\node (u'') at (2,0) [fixed,label=above:$u''$] {};
	\draw [red?] (u) -- node [above] {$k$} (u');
	\draw [blue?] (u') -- node [below] {} (u'');
	\draw [red?] (u) -- node [above] {$k$} (u'');
	%\draw (-.2,-.5) rectangle (2.2,.2);
\end{tikzpicture}
\end{center}
\label{fig5}
\caption{A depiction of the $(u,k)\in\A_{u'u'',k'}$.}
\end{figure}
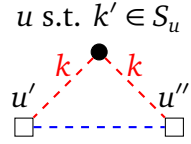
Based on our heuristics we predict the following trajectory of $\A_{u'u'',k'}$. First, we choose a vertex $u$ with two uncolored edges $uu'$ and $uu''$ having about $np^2$ choices. We need to make sure that $k'\in S_u$, which happens with probability $s$. The number of possible choices for $k\notin S_u \cup S_{u'}\cup S_{u''}$ is $|\COL|(1-s)^3$ and the probability that $k$ did not hit $u$, $u'$ or $u''$ in the previous steps is $p^3$. Finally, with probability $r^2$ we avoid $(uu', k)$- and $(uu'', k)$-alternating paths. Thus,
\[
\A_{u'u'',k'} \approx  |\COL| n s (1-s)^3 p^5r^2 = \frac 56 n^2 s (1-s)^2 p^5r^2 = n^2 a(t),
\]
where we define the scaled trajectory
\begin{equation}\label{eqn:atrajdef}
    a(t) := \frac 56  s (1-s)^2 p^5r^2.
\end{equation}

Next we define the type $B$ variables, which also fix two vertices and a color. For these variables we fix $u$, $u'$ and $k$. These are similar to the type $A$ variables but necessary due to the different roles the vertices and colors play in the process.

\begin{definition}
For each edge $uu'$ and each color $k \notin S_{u} \cup S_{u'}$ we define the random variable $\B_{uu',k}=\B_{uu',k}(i)$ to be the set of pairs $(u'', k')$ such that $k$ is available at $uu''$, $k'$ is available at  and $u'u''$,  and $k \in S_u''$.
\end{definition}
\begin{figure}[h!]
\begin{center}
\begin{tikzpicture}
	\node (u) at (0,0) [fixed,label=above:$u$] {};
	\node (u'') at (1,1) [vert,label=above:$u''$ s.t. $k \in S_u''$] {};
	\node (u') at (2,0) [fixed,label=above:$u'$] {};
	\draw [red?] (u) -- node [below] {} (u');
	\draw [blue?] (u') -- node [above] {$k'$} (u'');
	\draw [red?] (u) -- node [above] {$k$} (u'');
	%\draw (-.2,-.5) rectangle (2.2,.2);
\end{tikzpicture}
\end{center}
\label{fig6}
\caption{A depiction of the $(u'',k')\in\B_{uu',k}$.}
\end{figure}
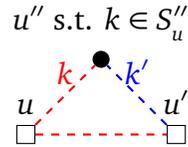
We heuristically predict that these have the same trajectory as the type $A$ variables. Indeed, the number of possible choices for $u''$ with uncolored $uu''$ and $u'u''$ is about $np^2$. The number of possible choices for $k'$ with $k'\in S_u$ and $k'\notin S_{u'}\cup S_{u''}$ is $|\COL|s(1-s)^2$ and the probability that $k\notin S_{u''}$ is $(1-s)$. Furthermore, the probability that color $k'$  did not hit neither $u'$ nor $u''$ is $p^2$ and the probability that $k$ did not hit $u''$ is $p$. Avoiding alternating paths $(uu'',k)$ and $(u'u'',k')$ is again a probability of $r^2$. Hence, 
\[
\B_{uu',k} \approx  |\COL| n s (1-s)^3 p^5r^2 = \frac 56 n^2 s (1-s)^2 p^5r^2 = n^2 b(t),
\]
where the scaled trajectory is 
\[
b(t) := \frac 56  s (1-s)^2 p^5r^2 = a(t).
\]

Next we define type $\CI{}$, $\CII{}$ and $\C_{uu'u''}$ variables which fix all of the vertices $u, u', u''$ and only count colors. 
\begin{definition}
For each ordered triple  $(u, u', u'')$ of uncolored edges  we define the random variable $\CI=\CI(i)$ to be the set of colors $k'$ such that $k'$ is 1-available at $(u, u', u'')$ at step $i$.  We define the random variable $\CII=\CII(i)$ to be the set of colors $k$ such that $k$ is 2-available at $(u, u', u'')$ at step $i$. We also define $\C_{uu'u''}(i)$ to be the set of pairs $(k, k')$ available at $(u, u', u'')$. In other words $\C_{uu'u''}(i)$ is the Cartesian product $\CI \times \CII$. 
\end{definition}
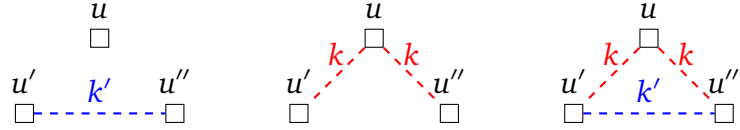
\begin{figure}[h!]
\begin{center}
\begin{tikzpicture}
	\node (u') at (0,0) [fixed,label=above:$u'$] {};
	\node (u) at (1,1) [fixed,label=above:$u$] {};
	\node (u'') at (2,0) [fixed,label=above:$u''$] {};
	\draw [blue?] (u') -- node [above] {$k'$} (u'');
\end{tikzpicture}
\hspace{4ex}
\begin{tikzpicture}
	\node (u') at (0,0) [fixed,label=above:$u'$] {};
	\node (u) at (1,1) [fixed,label=above:$u$] {};
	\node (u'') at (2,0) [fixed,label=above:$u''$] {};
	\draw [red?] (u) -- node [above] {$k$} (u');
	\draw [red?] (u) -- node [above] {$k$} (u'');
\end{tikzpicture}
\hspace{4ex}
\begin{tikzpicture}
	\node (u') at (0,0) [fixed,label=above:$u'$] {};
	\node (u) at (1,1) [fixed,label=above:$u$] {};
	\node (u'') at (2,0) [fixed,label=above:$u''$] {};
	\draw [blue?] (u') -- node [above] {$k'$} (u'');
	\draw [red?] (u) -- node [above] {$k$} (u');
	\draw [red?] (u) -- node [above] {$k$} (u'');
\end{tikzpicture}
\end{center}
\caption{Depictions of the $k'\in\CI$, $k\in\CII$, and $(k,k')\in C_{uu'u''}$.}
\end{figure}
Similarly, as for the previous variables we heuristically predict that 
\[
\CI \approx |\COL| s (1-s)^2p^2r = \frac{5}{6} n s (1-s)p^2r = n c_1(t), 
\]
\[
\CII \approx |\COL| (1-s)^3p^3r^2 = \frac{5}{6} n (1-s)^2p^3r^2 = n c_2(t),
\]
and
\[
\C_{uu'u''}(i) \approx \frac{25}{36}n^2 s(1-s)^3p^5r^3 = n^2 c(t),
\]
where the scaled trajectories are 
\[
c_1(t) := \frac{5}{6} s (1-s)p^2r, \quad c_2(t) := \frac{5}{6}  (1-s)^2p^3r^2 \quad\text{ and }\quad c(t) := \frac{25}{36} s(1-s)^3p^5r^3 = c_1(t) c_2(t).
\]

Now we have type $D, E$ and $F$ variables, where one vertex and one color are fixed. 

\begin{definition}
For each vertex  $u$ and each color $k \notin \S_u$ we define the random variable $\D_{u,k}=\D_{u,k}(i)$ to be the set of triples $(u', u'', k')$ such that $(k, k')$ is available at $(u, u', u'')$ at step $i$. 
\end{definition}

\begin{definition}
For each vertex  $u''$ and each color $k \notin \S_{u''}$ we define the random variable $\E_{u'',k}=\E_{u'',k}(i)$ to be the set of triples $(u, u', k')$ such that $(k, k')$ is available at $(u, u', u'')$ at step $i$. 
\end{definition}

\begin{definition}
For each vertex  $u''$ and each color $k' \notin \S_{u''}$ we define the random variable $\F_{u'',k'}=\F_{u'',k'}(i)$ to be the set of triples $(u, u', k)$ such that $(k, k')$ is available at $(u, u', u'')$ at step $i$. 
\end{definition}
Based on our heuristics we predict the following trajectories. Here, for example, we explain how to obtain the predicted trajectory of $\F_{u'',k'}$. First we choose an ordered pair $u$ and $u'$ with all uncolored edges. This gives us about $n^2p^3$ choices. Next we choose a color $k$ such that $k\notin S_u\cup S_{u'}\cup S_{u''}$ yielding $|\COL|(1-s)^3$ possibilities. Now we observe that the probability that $k'\in S_u$ and $k'\notin S_{u'}$ is $s(1-s)$. Furthermore, the probability that $u$, $u'$ and $u''$ are not hit by $k$ is $p^3$, and the probability that $k'$ did not hit $u'$ is $p$. Finally, the probability of avoiding alternating paths $(uu',k)$, $(uu'',k)$ and $(u'u'', k')$ is $r^3$. Thus, $\F_{u'',k'}  \approx  |\COL| n^2 s (1-s)^4p^7r^3$.

Trajectories of $\D_{u,k}$ and $\E_{u'',k}$ can be derived in a similar fashion. Consequently, 
\[
\D_{u,k}, \E_{u'',k}, \F_{u'',k'}  \approx  |\COL| n^2 s (1-s)^4p^7r^3 = \frac 56 n^3 s (1-s)^3p^7r^3
\]
with the scaled trajectories
\[
d(t), e(t), f(t) := \frac 56 s (1-s)^3p^7r^3.
\]

Finally we define our type $Z$ variables, which are useful for tracking which colors become forbidden due to alternating 4-cycles. In particular, for a fixed edge $uv$ and a color $k$, we keep track of substructures that could eventually cause $k$ to be forbidden at $uv$ due to a potential alternating 4-cycle.

\begin{definition}
Fix two vertices $u, v$, a color $k \notin S_u \cup S_v$ and a vector $(a_1, a_2, a_3) \in \{0, 1\}^3$ with $(a_1, a_2, a_3) \neq (1, 1, 1)$. 
%If $k$ is available at $uv$ at step $i$, 
We define the random variable $Z_{uv, k, a_1, a_2, a_3} = Z_{uv, k, a_1, a_2, a_3}(i)$ to be the number of triples $(x, y, k')$ where $x, y$ are vertices and $k'$ is a color satisfying the following condition. Letting $e_1:=ux$, $e_2:=xy$, $e_3:=yv$, and $k_1:=k'$, $k_2:=k, k_3:=k'$, we have for each $1 \le j \le 3$ that
\begin{enumerate}[label=$\bullet$]
    \item if $a_j=0$ then $k_j$ is available at $e_j$, and
    \item if $a_j=1$ then $e_j$ is assigned the color $k_j$.
\end{enumerate}
%Otherwise (i.e.\ if $k$ is not available at $uv$ at step $i$) we let $Z_{uv, k, a_1, a_2, a_3}(i)=Z_{uv, k, a_1, a_2, a_3}(i-1)$.
\end{definition}

\noindent
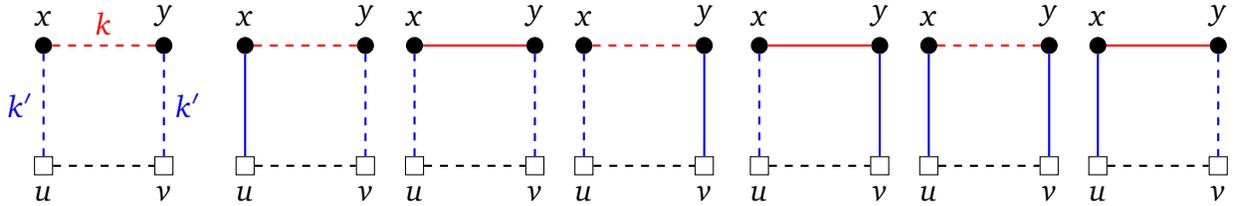
\begin{figure}[h!]
\begin{center}
\begin{tikzpicture}[scale=.8]
	\node (u) at (0,0) [fixed,label=below:$u$] {};
	\node (v) at (2,0) [fixed,label=below:$v$] {};
	\node (y) at (2,2) [vert,label=above:$y$] {};
	\node (x) at (0,2) [vert,label=above:$x$] {};
	\draw [uncolored] (u) -- (v);
	\draw [blue?] (y) -- node [right] {$k'$}(v);
	\draw [blue?] (u) -- node [left] {$k'$}(x);
	\draw [red?] (y) -- node [above] {$k$}(x);
\end{tikzpicture}
\hfill
\begin{tikzpicture}[scale=.8]
	\node (u) at (0,0) [fixed,label=below:$u$] {};
	\node (v) at (2,0) [fixed,label=below:$v$] {};
	\node (y) at (2,2) [vert,label=above:$y$] {};
	\node (x) at (0,2) [vert,label=above:$x$] {};
	\draw [uncolored] (u) -- (v);
	\draw [blue?] (y) -- (v);
	\draw [blue] (u) -- (x);
	\draw [red?] (y) -- (x);
\end{tikzpicture}
\begin{tikzpicture}[scale=.8]
	\node (u) at (0,0) [fixed,label=below:$u$] {};
	\node (v) at (2,0) [fixed,label=below:$v$] {};
	\node (y) at (2,2) [vert,label=above:$y$] {};
	\node (x) at (0,2) [vert,label=above:$x$] {};
	\draw [uncolored] (u) -- (v);
	\draw [blue?] (y) -- (v);
	\draw [blue?] (u) -- (x);
	\draw [red] (y) -- (x);
\end{tikzpicture}
\begin{tikzpicture}[scale=.8]
	\node (u) at (0,0) [fixed,label=below:$u$] {};
	\node (v) at (2,0) [fixed,label=below:$v$] {};
	\node (y) at (2,2) [vert,label=above:$y$] {};
	\node (x) at (0,2) [vert,label=above:$x$] {};
	\draw [uncolored] (u) -- (v);
	\draw [blue] (y) -- (v);
	\draw [blue?] (u) -- (x);
	\draw [red?] (y) -- (x);
\end{tikzpicture}
\hfill
\begin{tikzpicture}[scale=.8]
	\node (u) at (0,0) [fixed,label=below:$u$] {};
	\node (v) at (2,0) [fixed,label=below:$v$] {};
	\node (y) at (2,2) [vert,label=above:$y$] {};
	\node (x) at (0,2) [vert,label=above:$x$] {};
	\draw [uncolored] (u) -- (v);
	\draw [blue] (y) -- (v);
	\draw [blue?] (u) -- (x);
	\draw [red] (y) -- (x);
\end{tikzpicture}
\begin{tikzpicture}[scale=.8]
	\node (u) at (0,0) [fixed,label=below:$u$] {};
	\node (v) at (2,0) [fixed,label=below:$v$] {};
	\node (y) at (2,2) [vert,label=above:$y$] {};
	\node (x) at (0,2) [vert,label=above:$x$] {};
	\draw [uncolored] (u) -- (v);
	\draw [blue] (y) -- (v);
	\draw [blue] (u) -- (x);
	\draw [red?] (y) -- (x);
\end{tikzpicture}
\begin{tikzpicture}[scale=.8]
	\node (u) at (0,0) [fixed,label=below:$u$] {};
	\node (v) at (2,0) [fixed,label=below:$v$] {};
	\node (y) at (2,2) [vert,label=above:$y$] {};
	\node (x) at (0,2) [vert,label=above:$x$] {};
	\draw [uncolored] (u) -- (v);
	\draw [blue?] (y) -- (v);
	\draw [blue] (u) -- (x);
	\draw [red] (y) -- (x);
\end{tikzpicture}
\end{center}
\caption{Depictions of the $(x,y,k')\in Z_{uv,k,a_1,a_2,a_3}$ for $(a_1,a_2,a_3)\in \{(0,0,0), (1,0,0),(0,1,0),(0,0,1),(0,1,1),(1,1,0)\}$ (respectively).} 
\end{figure}

We anticipate the following trajectories. For example, we explain in detail how to predict $Z_{uv,k, 0,1,1}$. First we choose an ordered pair $x$ and $y$ such that $xy$ and $yv$ are already colored. For this we should have $n^2p(1-p)^2$ choices. Next we need to make sure that the color of $xy$ is $k$. This should happen with probability $1/|\COL|$. The color $k'$ is already determined by the color of $yv$ and $k'$ must be available at $ux$. In particular, $k'$ must not be in $S_u$ or $S_x$, which happens with probability $(1-s)^2$. Also $k'$ must not have already hit $u$ or $x$ before, which occurs with probability $p^2$. Finally there must not be an alternating $(ux,k')$-path, which happens with probability~$r$.  Thus, $Z_{uv,k, 0,1,1}  \approx  n^2p(1-p)^2 \cdot \frac{1}{|\COL|} \cdot (1-s)^2 \cdot p^2 \cdot r$. 

The remaining trajectories can be obtained similarly.
\begin{align*}
    Z_{uv,k, 0,0,0}  & \approx  |\COL| n^2 (1-s)^6 p^9 r^3 = \frac 56 n^3 (1-s)^5 p^9 r^3,\\
    Z_{uv,k, 1,0,0} \approx Z_{uv,k, 0,1,0} \approx Z_{uv,k, 0,0,1}   &\approx  n^2 (1-s)^4 (1-p) p^6 r^2,\\
    Z_{uv,k, 1,1,0} \approx Z_{uv,k, 1,0,1} \approx Z_{uv,k, 0,1,1}  &\approx  \frac{n^2}{|\COL|}  (1-s)^2 (1-p)^2 p^3 r = \frac65 n  (1-s)^3 (1-p)^2 p^3 r.
\end{align*}
Thus we define the following scaled trajectories: 
\[
z_0(t) := \frac 56 (1-s)^5 p^9 r^3, \quad z_1(t) := (1-s)^4 (1-p) p^6 r^2 \quad\text{ and }\quad z_2(t) := \frac65   (1-s)^3 (1-p)^2 p^3 r.
\]

\subsection{Derivatives of the trajectories}

First, we collect all the scaled trajectories:
\begin{align*}
    y(t) &= p^2,\\
    q(t) & = \frac16 p^3,\\
    a(t) = b(t) &= \frac 56  s (1-s)^2 p^5r^2,\\
    c_1(t) &= \frac{5}{6} s (1-s)p^2r,\\
    c_2(t) &= \frac{5}{6} (1-s)^2p^3r^2,\\
     d(t) = e(t) = f(t) &= \frac 56 s (1-s)^3p^7r^3,\\
     z_0(t) &= \frac 56 (1-s)^5 p^9 r^3,\\
     z_1(t) &= (1-s)^4 (1-p) p^6 r^2,\\
     z_2(t) & = \frac65  (1-s)^3 (1-p)^2 p^3 r.
\end{align*}
These functions satisfy the following system of differential equations. Each differential equation in the system naturally arises from estimating the expected one-step change in one of our random variables. The fact that our scaled trajectories satisfy this system is crucial to our analysis and will be used in our calculations. It is not hard to check (with, for example, a software such as Maple) that the system is satisfied using that $p'(t)=-6$ and $r'(t) = -\frac{648}{25}(1-s)^2(1-p)^2r$. We have:
\begin{align}
    a'(t) = b'(t) &= - \frac{5ad}{2qc}-  \frac{6a^2z_2}{qc}  - \frac{2ay}{q},\label{eqn:abdiffeq}\\
    c'_1(t) &= -\frac{5dc_1}{3qc}-\frac{3az_2c_1}{qc},\label{eqn:c1diffeq}\\
    c'_2(t) &= -\frac{5dc_2}{2qc}-\frac{6az_2c_2}{qc},\label{eqn:c2diffeq}\\
     d'(t) = e'(t) = f'(t) &=  -\frac{20d^2}{6qc}-\frac{9az_2d}{qc}-\frac{3yd}{q},\label{eqn:defdiffeq}\\
     z'_0(t) &= -\frac{5dz_0}{qc} - \frac{9az_2z_0}{qc} - \frac{3yz_0}{q},\label{eqn:z0diffeq}\\
     z'_1(t) &=  \frac{az_0}{qc} -\frac{10dz_1}{3qc} - \frac{6az_2z_1}{qc} - \frac{2yz_1}{q}, \label{eqn:z1diffeq}\\
     z'_2(t) & = \frac{2az_1}{qc}-\frac{5dz_2}{3qc} - \frac{3az_2^2}{qc} - \frac{yz_2}{q} \label{eqn:z2diffeq}.
\end{align}
We will also need a crude bound on the first and second derivatives of the scaled trajectories. Note that all these functions ($a, b$, etc.) have the form $h_1(t) \exp(h_2(t))$ where $h_1$ and  $h_2$ are polynomials.    It is easy to see that the derivative (and second derivative) of any such function has the form $h_3(t) \exp(h_2(t))$ where $h_3(t)$ is a polynomial. In particular, the first and second derivatives are all $O(1)$ for all $0 \le t \le 1$. Thus we have:
\begin{prop}\label{obs:CrudeDerivTraj}
 The first and second derivatives of all the scaled trajectory functions are $O(1)$.
\end{prop}

\subsection{Untracked variables}

In addition to the random variables we already mentioned, which we will track, we will also need several random variables for which we will will establish some necessary, but less precise, bounds in our analysis.   
In particular, when we consider the maximum one-step change in the $Z$ type variables, we could potentially lose a catastrophic number of triples through alternating paths forbidding edges in two types of pathological substructures.

\begin{definition}

\begin{enumerate}[label=(\roman*)]

\item  Fix two vertices $u,v$ and a color $k$. We define the random variable $\Xi_{u, v, k}  = \Xi_{u, v, k}(i)$  to be the number of pairs $(x,y)$ such that $ux$ has the same color as $vy$, and $xy$ has the color $k$. In other words, $\Xi_{u, v, k}$ is the number of alternating $(uv, k)$-paths. See Figure \ref{fig:xi}.

\item Fix four vertices $u,u',v,v'$ and a color $k$. We define the random variable $\Phi_{u, u', v, v'}  = \Phi_{u, u', v, v'}(i)$  to be the number of pairs $(x,y)$ such that $ux$ has the same color as $u'y$ and $vx$ has the same color as $v'y$. See Figure \ref{fig:phi}.

\item Fix two vertices $u,u''$ and colors $k, k''$. We define the random variable $\Psi_{u, u'', k, k''}  = \Psi_{u, u'', k, k''}(i)$ to be the number of triples $(x, y, z)$ such that $ux$ has the same color as $zu''$, $xy$ has the color $k$ and $yz$ has the color $k''$. See Figure \ref{fig:psi}.

\item Fix three vertices $u,v, w$. We define the random variable $\Lambda_{u,v, w}  = \Lambda_{u,v, w}(i)$ to be the number of pairs $(x, y)$ such that $ux$ has the same color as $vy$, and $vx$ has the same color as $wy$. See Figure \ref{fig:lambda}.

\end{enumerate}
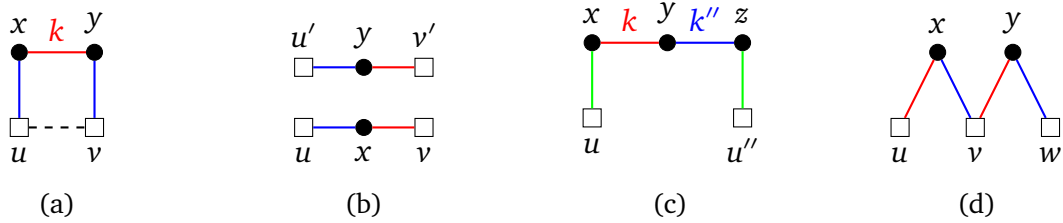
\begin{figure}[h!]
    \centering
    \begin{subfigure}[b]{0.24\textwidth}
    \centering
\begin{tikzpicture}[scale=0.5]
	\node (u) at (0,0) [fixed,label=below:$u$] {};
	\node (v) at (2,0) [fixed,label=below:$v$] {};
	\node (y) at (2,2) [vert,label=above:$y$] {};
	\node (x) at (0,2) [vert,label=above:$x$] {};
	\draw [uncolored] (u) -- (v);
	\draw [blue] (y) -- (v);
	\draw [blue] (u) -- (x);
	\draw [red] (y) -- node [above] {$k$} (x);
\end{tikzpicture}
    \caption{}
    \label{fig:xi}
    \end{subfigure}
    \begin{subfigure}[b]{0.24\textwidth}
    \centering
\begin{tikzpicture}[scale=.8]
	\node (u) at (0,0) [fixed,label=below:$u$] {};
	\node (u') at (0,1) [fixed,label=above:$u'$] {};
	
	\node (x) at (1,0) [vert,label=below:$x$] {};
	\node (y) at (1,1) [vert,label=above:$y$] {};
	
	\node (v) at (2,0) [fixed,label=below:$v$] {};
	\node (v') at (2,1) [fixed,label=above:$v'$] {};
	
	\draw [blue] (u) -- (x);
	\draw [blue] (u') -- (y);
	
	\draw [red] (x) -- (v);
	\draw [red] (v') -- (y);
\end{tikzpicture}
    \caption{}
    \label{fig:phi}
    \end{subfigure}
        \begin{subfigure}[b]{0.24\textwidth}
    \centering
\begin{tikzpicture}[scale=1]
	\node (u) at (0,0) [fixed,label=below:$u$] {};
	\node (u'') at (2,0) [fixed,label=below:$u''$] {};
	
	\node (x) at (0,1) [vert,label=above:$x$] {};
	\node (y) at (1,1) [vert,label=above:$y$] {};
	\node (z) at (2,1) [vert,label=above:$z$] {};
	
	\draw [green] (u) -- (x);
	\draw [red] (x) -- node [above] {$k$}(y);
	\draw [blue] (y) -- node [above] {$k''$}(z);
	\draw [green] (z) -- (u'');
\end{tikzpicture}
    \caption{}
    \label{fig:psi}
    \end{subfigure}
        \begin{subfigure}[b]{0.24\textwidth}
    \centering
\begin{tikzpicture}
	\node (u) at (0,0) [fixed, label=below:$u$] {};
	\node (v) at (1,0) [fixed, label=below:$v$] {};
	\node (w) at (2,0) [fixed, label=below:$w$] {};
	\node (x) at (.5,1) [vert, label=above:$x$] {};
	\node (y) at (1.5,1) [vert, label=above:$y$] {};
	\draw [red] (u) -- (x);
	\draw [red] (v) -- (y);
	\draw [blue] (v) -- (x);
	\draw [blue] (w) -- (y);
\end{tikzpicture}
    \caption{}
    \label{fig:lambda}
    \end{subfigure}
    \caption{Depictions of the $(x,y)\in\Xi_{u, v, k}$,~$(x,y)\in\Phi_{u, u', v, v'}$,~$(x,y,z)\in\Psi_{u, u'', k, k''}$ and $(x,y)\in \Lambda_{u,v, w}$.}
\end{figure}
\end{definition}

\section{The good event}\label{sec:goodevent}

In this section we define  the good event $\mc{E}_i$, which among other things stipulates that every uncolored triangle $(u, u', u'')$ still has plenty of available pairs of colors $(k, k')$.  More specifically, $\mc{E}_i$ will stipulate that all of our tracked variables are within a small window of their respective trajectories we derived in Section \ref{sec:variables}. The event $\mc{E}_i$ will also stipulate some crude upper bounds on certain other variables. Note that in the process, we choose $\varepsilon$, which gives us $s(\varepsilon)$. Then we let 
\[
\delta:=10^{-7}s(1-s)^4
\]
and define below all of the error functions $g_q, g_y$, etc..

For any step $i'$ we let $t'=t(i')$. We formally define the \tbf{good event} $\mc{E}_i$ to be the event that for all $i' \le i$ we have the following conditions (below, functions are evaluated at $i=i'$, $t=t'$):
 \begin{enumerate}[label=(\Roman*)]
 \item \label{E:Q} we have
     \[
     \abrac{Q -  n^3 q(t)} \le n^{3} g_q,
     \]
     \item \label{E:Y} for each uncolored edge $uu'$ we have
     \[
     \abrac{Y_{uu'} - ny(t)} \le n g_y,
     \]
     \item \label{E:A} for each uncolored edge $u'u''$ we have
     \[
     \abrac{A_{u'u'',k'} -n^2a(t)} \le n^2 g_{ab},
     \]
     \item \label{E:B} for each uncolored edge $uu'$ and color $k$ we have
     \[
     \abrac{B_{uu',k} -n^2 b(t)} \le n^2 g_{ab},
     \] 
     \item \label{E:C1} for each triple $(u, u', u'')$ of uncolored edges we have
     \[
     \abrac{\CI -nc_1(t)} \le n g_{c1},
     \]
     \item \label{E:C2} for each triple $(u, u', u'')$ of uncolored edges we have
     \[
     \abrac{\CII -nc_2(t)} \le n g_{c2},
     \]
     \item \label{E:D} for each vertex $u$ and color $k$ available at $u$ we have
     \[
     \abrac{D_{u,k} -n^3d(t)} \le n^3 g_{def},
     \]
     \item \label{E:E} for each vertex $u''$ and color $k$ available at $u''$ we have
     \[
     \abrac{E_{u'',k} -n^3e(t)} \le n^3 g_{def},
     \]
     \item \label{E:F} for each vertex $u''$ and color $k'$ available at $u''$ we have 
     \[
     \abrac{F_{u'',k'} -n^3f(t)} \le n^3 g_{def},
     \]
     \item \label{E:Z0} for each uncolored edge $uv$ and color $k$ we have 
     \[
     \abrac{Z_{uv,k, 0,0,0} -n^3z_0(t)} \le n^3 g_0,
     \]
     \item \label{E:Z1} for each uncolored edge $uv$ and color $k$ we have
     \[
     \abrac{Z_{uv,k, 1,0,0} -n^2z_1(t)} \le n^2 g_1,
     \]
    \[
     \abrac{Z_{uv,k, 0,1,0} -n^2z_1(t)} \le n^2 g_1,
     \]
    \[
    \abrac{Z_{uv,k, 0,0,1} -n^2z_1(t)} \le n^2 g_1,
    \]
    \item \label{E:Z2} for each uncolored edge $uv$ and color $k$ we have
    \[
    \abrac{Z_{uv,k, 1,1,0} -nz_2(t)} \le n g_2,
    \]
    \[
    \abrac{Z_{uv,k, 1,0,1} -nz_2(t)} \le n g_2,
    \]
    \[
    \abrac{Z_{uv,k, 0,1,1} -nz_2(t)} \le n g_2.
    \]
    
        \item \label{E:crude}
    for all $u, v, k$ we have
    \[
    \Xi_{u, v, k} \le n^{4\delta},
    \]
    for all $u,u', v,v'$ we have
    \[
    \Phi_{u, u', v, v'} \le n^{4\delta},
    \]
     for all $u, u'', k,k'$ we have
    \[
    \Psi_{u, u'', k, k''} \le n^{4\delta},
    \]
    and for all $u, v, w$ we have
    \[
    \Lambda_{u, v, w} \le n^{4 \delta}.
    \]
 \end{enumerate}

Recall that we define the random variable $\C_{uu'u''}=\C_{uu'u''}(i)$ as $\CI \times \CII$. This will count the number of pairs $(k,k')$ that are available at $(u, u', u'')$ at step $i$. In addition, we let 
\[
c(t) := c_1(t)c_2(t) \quad\text{and}\quad g_c:=2(c_2g_{c_1}+c_1g_{c_2}).
\] 
Since $g_{c_1}=o(c_1)$, we get in the good event,
\[
    \C_{uu'u''}\le n(c_1+g_{c_1}) \cdot n(c_2 + g_{c_2})
    = n^2(c + c_1 g_{c_2} + g_{c_1}c_2 + g_{c_1}g_{c_2})
    \le n^2(c + g_{c})
\]
and similarly $\C_{uu'u''}\ge n^2(c - g_c)$. Thus,
\[
     \abrac{\C_{uu'u''} -n^2c(t)} \le n^2 g_{c}.
\]

We let
 \[
 i_{max}:= \frac16 n^2 \rbrac{1-n^{-\d}}, \qquad t_{max} := \frac{i_{max}}{n^2} = \frac16 \rbrac{1-n^{-\d}},
 \]
 and note that 
 \[
 p(t_{max}) = 1-6t_{max} = n^{-\d}.
 \]
Let 
\[
\kappa = \kappa(s) = 10000 s^{-1}(1-s)^{-4} \qquad \textrm{and}\qquad\omega = 100(\kappa + 1)\delta.
\]
Note that $\omega = 100 \delta + 1/10 = 10^{-5}s(1-s)^4 + 1/10< 1/4$, since $s=s(\eps)$ and $\eps <1/100$. 
We define the error functions as follows:
 \begin{align*}
g_y(t) &=  n^{-1/2+\delta},\\
g_q(t) &=  n^{-1+2\delta},\\
g_{ab}(t) &= n^{-\omega} p(t)^{-100 \kappa}, \\ 
g_{c_1}(t) &= n^{-\omega} p(t)^{-100 \kappa-2}, \\ 
g_{c_2}(t) &= n^{-\omega} p(t)^{-100 \kappa-1}, \\ 
g_{def}(t) &= n^{-\omega} p(t)^{-100 \kappa+3}, \\ 
g_{0}(t) &= n^{-\omega} p(t)^{-100 \kappa+5}, \\ 
g_{1}(t) &= n^{-\omega} p(t)^{-100 \kappa+1}, \\ 
g_{2}(t) &= n^{-\omega} p(t)^{-100 \kappa-1}.  
\end{align*}

Note that
$$
\frac{g_q(t)}{q(t)} = \frac{n^{-1+2\delta}}{\frac16 p(t)^3} \le 6 n^{-1+5\delta} = o(1) \qquad %\textrm{and} \qquad \frac{g_y(t)}{y(t)} = \frac{n^{-1+2\delta}}{\frac16 p(t)^3} \le 6 n^{-1+5\delta} = o(1)
$$
since $p \ge n^{-\delta}$ and $\delta < 1/1000$. 
Furthermore, using these and that $\eps/2 \le s \le 3\eps/5$ and $r\ge 1/6$, we obtain
$$
\frac{g_{c}}{c} = \frac{2(c_2g_{c_1}+c_1g_{c_2})}{c} = O\rbrac{\frac{p^3 \cdot n^{-\omega} p(t)^{-100 \kappa-2} + p^2 \cdot n^{-\omega} p(t)^{-100 \kappa-1}}{p^5} }  = O\rbrac{n^{-\omega + (100\k-4)\d} } = o(1).
$$
%In addition, one may check that the other error functions are also of smaller order than their trajectories. The only technicality is that $ z_1$ and $z_2$ may be zero at the beginning of the process, but this will not affect our analysis.  

It is also routine to check that the error function satisfy the following, which will be required at a crucial point in our analysis:
\begin{align}
& g_{ab}' - 30 \k \rbrac{p^{2} g_2 + p^{-1}g_{ab}  + p^{-1} g_c} \label{eq:err-sup-ab}\\
    &\qquad\qquad = n^{-\omega}\left(570{p}^{-100\,\kappa-1} -30{p}^{-100\kappa+1}-120{p}^{-100\,\kappa}\right)
=\Omega(1), \nn \\
& g_{c_1}' - 30 \k (p^{-1}g_2 + p^{-3}g_{ab} +  p^{-4}g_c+p^{-6}g_{def}) \label{eq:err-sup-c1}\\
    &\qquad\qquad= n^{-\omega} \left( \left( 420\,\kappa+12 \right) {p}^{-100\,\kappa-3}-30\,\kappa\,{p}^{-
100\,\kappa-2} \right) =\Omega(1),\nn\\
&g_{c_2}' - 30 \k (p^{-1}g_2 + p^{-2}g_{ab} + p^{-3}g_c+p^{-5}g_{def}) \label{eq:err-sup-c2}\\
    &\qquad\qquad= n^{-\omega}\left( 390\,\kappa+6 \right) {p}^{-100\,\kappa-2}
=\Omega(1),\nn\\
& g_{def}' - 30 \kappa \rbrac{p^{4}g_2 + p^{-1}g_{def} + p^{2}g_{ab}+pg_c} \label{eq:err-sup-def}\\
    &\qquad\qquad =  n^{-\omega} \left(\left( 420\,\kappa-18 \right) {p}^{-100\,\kappa+2}-30\,\kappa\,{p}^{-
100\,\kappa+3} \right)
=\Omega(1),\nn\\
& g_{0}' - 30 \k \rbrac{p^{6}g_2 + pg_{def} + p^{4}g_{ab}+p^{3}g_c}\label{eq:err-sup-0}\\
    &\qquad\qquad=  n^{-\omega} \left(\left( 420\,\kappa-30 \right) {p}^{-100\,\kappa+4}-30\,\kappa\,{p}^{-
100\,\kappa+5} \right)
=\Omega(1),\nn\\
& g_{1}' - 40 \k \rbrac{p^{3}g_2 + p^{-2}g_{def} + pg_{ab}+p^{-1}g_c}\label{eq:err-sup-1}\\
    &\qquad\qquad=  n^{-\omega} \left(\left( 440\,\kappa-6 \right) {p}^{-100\,\kappa}-80\,\kappa\,{p}^{-100\,\kappa+1}-40\,{p}^{-100\,\kappa+2}\kappa \right)
=\Omega(1),\nn\\
& g_{2}' - 40 \k \rbrac{p^{-1}g_2 + p^{-5}g_{def} + p^{-2}g_{ab}+p^{-3}g_c} \label{eq:err-sup-2}\\
    &\qquad\qquad=  n^{-\omega} \left( 320\,\kappa+6 \right) {p}^{-100\,\kappa-2}
=\Omega(1).\nn
\end{align}

Finally note that all error functions have the form $n^{-\omega} p^{-h}$ where $h \ge -100\k-2$ is a constant. So the first derivative is a constant times $n^{-\omega}p^{-h-1}$ and the second derivative is a constant times $n^{-\omega}p^{-h-2}$.
In particular for all $0 \le t \le 1$ the second derivative of any error function is
\[
O\rbrac{n^{-\omega}p^{-100\k-4} } = O\rbrac{n^{-\omega + (100\k+4)\d} }
\]
and similarly for the first derivative. Thus we have:
\begin{prop}\label{obs:CrudeDerivErr} 
 The first and second derivatives of all the error functions are $O(1)$.
\end{prop}
 
\section{Some helpful bounds that hold in the good event}\label{sec:helpers}

Some of these bounds will be sharp and others will be more crude upper bounds.

\subsection{Sharp estimates}

In order to track certain variables we will need to estimate the probabilities of the following events. 
\begin{enumerate}[label=$\bullet$]
\item For a color $k^*$ available at $v$ at step $i$, we let $\mc{L}_{v,k^*}=\mc{L}_{v,k^*}(i)$ be the event that $v$ gets hit by $k$ at step~$i$.
\item For a color $k^*$ available at edge $vw$, we let $\mc{M}_{vw,k^*}=\mc{M}_{vw,k^*}(i)$ be the event that $vw$ becomes the color $k^*$ at step~$i$. 
\item For an uncolored edge $vw$, we let $\mc{M}_{vw,\bullet}=\mc{M}_{vw,\bullet}(i)$ be the event that $vw$ gets colored any color at step~$i$.
\item For any color $k$ and an uncolored edge $uv$, let $\mc{N}_{uv,k}=\mc{N}_{uv,k}(i)$ be the event that the edge $uv$ will become part of a $(uv,k)$-alternating path.
\end{enumerate}

\begin{claim}\label{claim:LMMN}
Assuming the good event $\mc{E}_i$ holds, we have:
\begin{enumerate}[label=(\roman*)]
\item\label{claim:LMMN:L} $\displaystyle{n^{-2} \sbrac{ \frac{5d}{6qc}  - \k \rbrac{p^{-8}g_{def}+ p^{-6} g_c}} \le \P(\mc{L}_{v,k^*}) \le n^{-2} \sbrac{ \frac{5d}{6qc}  + \k \rbrac{p^{-8}g_{def} + p^{-6} g_c}}}$,
\item\label{claim:LMMN:M}  $\displaystyle{n^{-3}  \sbrac{\frac{a}{qc}   - \k p^{-8} \rbrac{g_{ab} + g_c}} \le \P(\mc{M}_{vw,k^*}) \le n^{-3}  \sbrac{\frac{a}{qc}   + \k p^{-8} \rbrac{g_{ab} + g_c}}}$,
\item\label{claim:LMMN:M2} $\displaystyle{n^{-2} \frac{y}{q}  - O(n^{-5/2+4\d}) \le \P(\mc{M}_{vw,\bullet}) \le  n^{-2} \frac{y}{q}  + O(n^{-5/2+4\d})}$, and
\item\label{claim:LMMN:N} $\displaystyle{n^{-2}  \sbrac{\frac{3az_2}{qc}   - \k \rbrac{p^{-3} g_2 + p^{-5}g_{ab} + p^{-5} g_c}} \le \P(\mc{N}_{uv,k}) \le n^{-2}  \sbrac{\frac{3az_2}{qc}   + \k \rbrac{p^{-3} g_2 + p^{-5}g_{ab} + p^{-5} g_c}}}$.
\end{enumerate}
\end{claim}

We also state some simple bounds related to the above probabilities. These bounds easily follow from the trajectories given in Section \ref{sec:variables}, $0 \le p \le 1$, $1/5 \le r \le 1$ and the assumption that $s>0$ is sufficiently small. Therefore, the proof is omitted.
\begin{claim}\label{claim:helpersloppy}
We have:
\[
    \frac{d}{qc} \le 50p^{-1}, \qquad \frac{az_2}{qc} \le 10, \qquad \frac{a}{qc} \le 10p^{-3}, \qquad\text{and}\qquad \frac{y}{q} \le 10p^{-1}.
\]
\end{claim}

We will use the following upper and lower bounds throughout the proof of Claim~\ref{claim:LMMN}. The following fact is easily checked, and thus we omit the proof. 
\begin{fact}\label{fact:sharp-bnd}
Let $x=x(n), y=y(n), z=z(n)$ with $x,y,z \in (0,1)$ and $y, z, =o(1)$. Then, for sufficiently large $n$, we have
$$
\frac{1+x}{(1-y)(1-z)} \le 1 + 2x + 2y + 2z \qquad \textrm{and} \qquad \frac{1-x}{(1+y)(1+z)} \ge 1 - 2x - 2y- 2z.
$$
\end{fact}

\begin{proof}[Proof of Claim~\ref{claim:LMMN}] We prove each statement separately.

\bigskip
\noindent
\emph{Part~\ref{claim:LMMN:L}:} By using Fact~\ref{fact:sharp-bnd} we get 
\begin{align*}
 \P(\mc{L}_{v,k^*}) &= \frac12 \sum_{(u', u'', k')\in \D_{v,k^*}}  \frac{1}{3QC_{vu'u''}}
	+ \sum_{(u, u', k') \in \E_{v,k^*}}  \frac{1}{3Q C_{uu'v}}
	+ \sum_{(u, u', k)\in \F_{v,k^*}}  \frac{1}{3Q C_{uu'v}}\\
	& \le \frac52 \cdot \frac{n^3(d +  g_{def})}{3 n^3\rbrac{q - g_q}  \cdot n^2 \rbrac{ c -  g_c}}
	 = n^{-2} \cdot \frac{5d}{6qc} \cdot  \frac{1 + \frac{g_{def}}{d}}{\rbrac{1 - \frac{g_q}{q}} \rbrac{1 - \frac{g_c}{c }}}.  
\end{align*}
Since $g_q/q$ and $g_c/c$ are $o(1)$, we use Fact~\ref{fact:sharp-bnd} and next $g_q/q = o(g_{def}/d)$ to obtain 
\begin{align*}	
 \P(\mc{L}_{v,k^*})& \le n^{-2} \cdot \frac{5d}{6qc} \rbrac{1   + \frac{2g_{def}}{d} + \frac{2g_q}{q}+ \frac{2g_c}{c }}
 \le n^{-2} \cdot \frac{5d}{6qc} \rbrac{1   + \frac{4g_{def}}{d} + \frac{2g_c}{c }}\\
	&= n^{-2} \rbrac{\frac{5d}{6qc}+\frac{144}{5}\frac{1}{(1-s)^3sr^3}p^{-8}g_{def}+ \frac{432}{25}\frac{1}{(1-s)^3sr^3}p^{-6}g_c}.
\end{align*}
Finally, since $r^{-3} \le 5^3$ and $\kappa(s) = 10^4s^{-1}(1-s)^{-4}$, we obtain the required upper bound on $\P(\mc{L}_{v,k^*})$. Using a similar calculation and the lower bound in Fact~\ref{fact:sharp-bnd} gives the lower bound.

\bigskip
\noindent
\emph{Part~\ref{claim:LMMN:M}:} Using similar calculations as Part~\ref{claim:LMMN:M} yields
\begin{align*}
   \P(\mc{M}_{vw,k^*}) &= 
\sum_{(u, k)\in \A_{vw,k^*}} \frac{1}{3Q C_{uvw}}
+ \sum_{(u'', k')\in \B_{vw,k^*}}  \frac{1}{3QC_{vwu''}}
+ \sum_{(u'',k')\in \B_{wv,k^*}}  \frac{1}{3QC_{wvu''}}\\
    & \le 3 \cdot \frac{ n^2(a + g_{ab})}{3 n^3\rbrac{q-g_q}  \cdot n^2 \rbrac{ c- g_c}}
     = n^{-3} \frac{a}{qc} \cdot  \frac{1 + \frac{g_{ab}}{a}}{\rbrac{1 - \frac{g_q}{q}} \rbrac{1 - \frac{g_c}{c }}}\\
    &\le  n^{-3} \cdot \frac{a}{qc} \rbrac{1   + \frac{2g_{ab}}{a} + \frac{2g_q}{q}+ \frac{2g_c}{c }}
    \le  n^{-3} \cdot \frac{a}{qc} \rbrac{1   + \frac{4g_{ab}}{a} + \frac{2g_c}{c }}\\
    &= n^{-3}\rbrac{\frac{a}{qc}+\frac{864}{25}\frac{1}{(1-s)^3 s r^3}p^{-8}g_{ab} + \frac{2592}{125}\frac{1}{(1-s)^4sr^4}p^{-8}g_c}\\
    & \le n^{-3}  \sbrac{\frac{a}{qc}   + \k p^{-8} \rbrac{g_{ab} +  g_c}}, 
\end{align*}
where for the latter term we use the fact that $r \ge \exp \cbrac{ - \frac{36}{25}}$.

\bigskip
\noindent
\emph{Part~\ref{claim:LMMN:M2}:} We have
\begin{align*}
   \P(\mc{M}_{vw,\bullet}) &= \frac{Y_{vw}}{Q}
    \le \frac{ny + n^{1/2+\d}}{n^3 q - n^{2+2\d}} 
    = n^{-2}  \frac{y}{q} \left( \frac{1 + n^{-1/2+\d}\frac 1y }{ 1 - n^{-1+2\d}\frac 1q} \right) 
    =  n^{-2}   \frac{y}{q}\left( \frac{1 + n^{-1/2+3\d} }{ 1 - 6n^{-1+5\d}} \right)\\
   & =  n^{-2} \frac{y}{q} \rbrac{1 + O(n^{-1/2+3\d}) }
    =  n^{-2} \frac{y}{q}  + O(n^{-5/2+4\d}).
\end{align*}

\bigskip
\noindent
\emph{Part~\ref{claim:LMMN:N}:} Finally note that
\begin{align*}
\P(\mc{N}_{uv,k}) &= 
\sum_{(x,y,k') \in Z_{uv,k, 1,0,1}} \P(\mc{M}_{xy,k} )
+ \sum_{(x,y,k') \in Z_{uv,k, 0,1,1}} \P(\mc{M}_{ux,k'})
+ \sum_{(x,y,k') \in Z_{uv,k, 1,1,0}} \P(\mc{M}_{vy,k'})\\
    & \le 3 \cdot n \rbrac{z_2 + g_2}\cdot n^{-3}  \sbrac{\frac{a}{qc}   + \k \rbrac{p^{-8}g_{ab} +  p^{-8} g_c}}\\
    & \le n^{-2}  \sbrac{\frac{3az_2}{qc}   + \k \rbrac{p^{-3} g_2 + p^{-5}g_{ab} + p^{-5} g_c}},
\end{align*}
where in the latter we use the fact that $g_2 \le 10z_2/11$ (since $g_2=o(z_2)$) and $z_2 \le p^3/3$.
\end{proof}

\subsection{Crude upper bounds}

Here we will (crudely) bound probabilities of intersections of the events defined in the previous subsection. 

\begin{claim}\label{claim:intersections}
The following holds in the good event $\mc{E}_i$.

\begin{enumerate}[label=(\roman*)]
\item\label{claim:intersections:LL} Fix vertices $v \neq v'$ and colors $k$ and $k'$ (where we allow $k=k'$). Then, 
\[
\P(\mc{L}_{v,k}(i) \cap \mc{L}_{v',k'}(i)) = O\rbrac{n^{-3 + 8\d}}.
\]

\item\label{claim:intersections:LM} Fix a vertex $v$, colors $k$ and $k'$ and an edge $e$. 
\begin{enumerate}[label=$\bullet$]
\item If $k' \neq k$, then 
   \[
\P(\mc{L}_{v,k}(i) \cap \mc{M}_{e, k'}(i)) = O\rbrac{n^{-4 + 8\d}}.
\]
\item If $v$ is not incident with $e$, then 
   \[
\P(\mc{L}_{v,k}(i) \cap \mc{M}_{e, k'}(i)) \le \P(\mc{L}_{v,k}(i) \cap \mc{M}_{e, \bullet}(i)) = O\rbrac{n^{-4 + 8\d}}.
\]
\item If $v$ is incident with $e$, then
    \[
\P(\mc{L}_{v,k}(i) \cap \mc{M}_{e, k'}(i)) \le \P(\mc{L}_{v,k}(i) \cap \mc{M}_{e, \bullet}(i)) = O\rbrac{n^{-3 + 8\d}}.
\]
\end{enumerate}

\item\label{claim:intersections:MM} Fix distinct (but possibly adjacent) edges $e$ and $e'$ and a color $k$. Then, 
\[
\P(\mc{M}_{e, k}(i) \cap \mc{M}_{e', \bullet}(i)) = O\rbrac{n^{-4 + 8\d}}
\quad \text{ and }\quad
\P(\mc{M}_{e, \bullet}(i) \cap \mc{M}_{e', \bullet}(i)) = O\rbrac{n^{-3 + 8\d}}.
\]

\item\label{claim:intersections:LN} Fix a vertex $v$, colors $k$ and $k'$ (possibly equal) and an edge $e$ (possibly incident with $v$). Then, 
\[
\P(\mc{L}_{v,k}(i) \cap \mc{N}_{e, k'}(i)) = O\rbrac{n^{-3 + 8\d}}.
\]

\item\label{claim:intersections:NN} Fix distinct (but possibly adjacent) edges $e, e'$, and colors $k, k'$ (possibly equal). Then, 
\[
\P(\mc{N}_{e, k} \cap \mc{N}_{e', k'}) = O\rbrac{n^{-3 + 12\d}}.
\]

\item\label{claim:intersections:NM} Fix edges $e, e'$ and colors $k, k'$.
\begin{enumerate}[label=$\bullet$]
\item If we assume nothing about $e, e'$ being distinct or nonadjacent or $k, k'$ being distinct, then,
\[
\P(\mc{N}_{e, k}  \cap \mc{M}_{e', k'}) \le \P(\mc{N}_{e, k}  \cap \mc{M}_{e', \bullet}) = O\rbrac{n^{-3 + 8\d}}.
\]
\item If $k \neq k'$ and $e \neq e'$ are adjacent, then
\[
\P(\mc{N}_{e, k}  \cap \mc{M}_{e', k'}) = O\rbrac{n^{-4 + 8\d}}.
\]
\item Suppose $e=uv$ and  $e'=xy$ are distinct and nonadjacent, and $k=k'$. Suppose further that it is not the case that $ux$ has the same color as $vy$ or that $uy$ has the same color as $vx$.  Then,
\[
\P(\mc{N}_{e, k}  \cap \mc{M}_{e', k'}) = O\rbrac{n^{-4 + 8\d}}.
\]
\end{enumerate}

\end{enumerate}
\end{claim}

\begin{proof} The above bounds have fairly straightforward proofs and, therefore, we omit most of them. Here we only show details for bounds in Parts \ref{claim:intersections:LL} and~\ref{claim:intersections:LN}.

We start with the following observation. Fix any oriented triangle $(v, v', v'')$ and pair of colors $(k, k')$. The probability that at step $i$ we choose $(v, v', v'')$ to color, and then choose the color pair $(k, k')$ to use, is 
\begin{equation}\label{eqn:Pcrude}
    \frac{1}{3Q} \cdot \frac{1}{C_{vv'v''}} = O\rbrac{\frac{1}{n^3q \cdot n^2 c}}   = O\rbrac{\frac{1}{n^3p^3 \cdot n^2 p^5}} = O\rbrac{n^{-5 + 8\d}}.
\end{equation}

\bigskip
\noindent
\emph{Part~\ref{claim:intersections:LL}:}
For the event $\mc{L}_{v,k} \cap \mc{L}_{v',k'}$ to happen, the triangle chosen at step $i$ must contain both $v$ and $v'$ and so there are a linear number of choices for the triangle. The color pair must include $k$ so there is at most a linear number of choices for the color pair. Since each possibility occurs with probability at most $O\rbrac{n^{-5 + 8\d}}$ by \eqref{eqn:Pcrude}, we have
\[
\P(\mc{L}_{v,k}(i) \cap \mc{L}_{v',k'}(i)) = O(n) \cdot O(n) \cdot O\rbrac{n^{-5 + 8\d}} = O\rbrac{n^{-3 + 8\d}}.
\]

\bigskip
\noindent
\emph{Part~\ref{claim:intersections:LN}:} Suppose $e$ is an uncolored edge at step $i$. Let $P(e, k)=P(e, k, i)$ be the set of all pairs $(e^*, k^*)$ where $e^*$ is an edge and $k^*$ is a color such that coloring $e^*$ the color $k^*$ would forbid $k$ at $e$ through an alternating 4-path. More precisely, if $e=wx$ then $P(e, k)$ is the following set:
\begin{align*}
     &\{(yz, k): \mbox{$yz$ is not adjacent to $wx$, and $wy$  has the same color as $zx$} \} \\ 
     &\qquad \cup \{(wy, k''):  \mbox{ there exists some $z$ where $yz$ has color $k$ and $zx$ has color $k''$}\} \\
     &\qquad\qquad \cup \{(xz, k''):  \mbox{ there exists some $y$ where $wy$ has color $k''$ and $yz$ has color $k$}\}.
\end{align*}

We split the proof into cases. Due to~\eqref{eqn:Pcrude} it suffices to show that the number of choices for the triangle (containing $v$) and colors (containing $k$) is at most $O(n^2)$. In order for $\mc{N}_{e, k'}$ to happen there must be some pair $(e^*, k^*) \in P(e, k')$ such that $e^*$ gets assigned the color $k^*$.

Suppose that $e^*$ is adjacent to $e$ and $k^*=k$. Since no vertex is adjacent to more than two edges of the same color, there are $O(1)$ choices for $(e^*,k^*)$ with this property. There are $O(n)$ triangles containing $e^*$, and $O(n)$ ways to choose the other color in the color pair. 

Now suppose that $e^*$ is adjacent to $e$ and $k^*\neq k$. There are $O(n)$ choices for $(e^*,k^*)$ with this property. There are $O(n)$ triangles containing $e^*$, and the color pair must consist of $k$ and $k^*$. 

Next assume that $e^*$ is not adjacent to $e$ and does not contain $v$. There are $O(n)$ choices for $e^*$, and once we choose one the triangle is determined. One color must be $k$ and we have $O(n)$ choices for the other color. 

Finally assume that $e^*$ is not adjacent to $e$ and contains $v$. There are $O(1)$ choices for $e^*$, and so $O(n)$ choices for the triangle. One color must be $k$ and we have $O(n)$ choices for the other color. This completes the proof of Part~\ref{claim:intersections:LN}.

As we mentioned, the proofs for the rest of the parts of the claim are very similar and the reader can easily check them. Some of them use the bounds in \ref{E:crude}.
\end{proof}

\section{Variables \texorpdfstring{$Q$ and $Y$}{}}\label{sec:QY}

We now begin verifying that the good event holds, starting with~\ref{E:Q} and~\ref{E:Y}. Both of the variables $Q$ and~$Y$ were tracked by Bohman, Frieze and Lubetzky in~\cite{BFL10}, so we will use a weaker form of their results. 
They showed that  a.a.s.\ for all 
\[
i \le i_{0} = \frac 16 n^2 - \frac 53 n^{7/4} \log^{5/4} n
\]
we have 
$$
n^3q(t) - n^2 \log n \cdot \frac{(5-30\log p(t))^2}{p(t)} \le Q(i) \le  n^3q(t) + \frac{1}{3}n^2p(t), \textrm{ and}
$$
$$
|Y_{uu'}-y(t)n|\le \sqrt{n\log n} \cdot (5-30\log p(t)) \textrm{ for all } uu'.
$$
These bounds are better than we need so we will loosen and simplify them. Note that as long as $\d<1/4$ we have that $i_{max} \le i_0$. Thus, using that $p(t) \ge p(t_{max}) = n^{-\d}$ the above bounds on $Q$ and $Y$ imply that for all $i \le i_{max}$ that 
\[
\abrac{Q -  n^3 q(t)} \le n^{2+2\d}
\]
and
\[
\abrac{Y_{uu'}-y(t)n} \le n^{1/2+\d}.
\]
Thus,
$\mc{E}_{i_{max}}$ a.a.s.\ does not fail due to conditions  \ref{E:Q} or \ref{E:Y}.

\section{Variable \texorpdfstring{$A$}{}}\label{sec:AB}

In this section we bound the probability that $\mc{E}_{i_{max}}$ fails due to a variable of type $A$ straying too far from its trajectory and violating Condition \ref{E:A}. Several of the sections that follow will have a very similar structure, so we will explain our reasoning carefully in this section so we can go faster in future sections. 
In addition, we will only show details for representatives of the four following groups of variables 
\[
A_{u'u'',k'} \in \{A_{u'u'',k'},B_{uu',k}\}, \qquad
C^{(1)}_{uu'u''}\in\{C^{(1)}_{uu'u''}, C^{(2)}_{uu'u''}\},\qquad 
D_{u,k} \in \{D_{u,k}, E_{u'',k}, F_{u'',k'}\}
\]
and
\[
Z_{uv,k,0,0,0}\in\{Z_{uv,k,0,0,0},
Z_{uv,k,1,0,0},Z_{uv,k,0,1,0},Z_{uv,k,0,0,1},Z_{uv,k,0,1,1},Z_{uv,k,1,0,1},Z_{uv,k,1,1,0}\}.
\]
The variables within a group require similar calculations, and in some cases have the exact same trajectory. In the case of the $Z$ variables, extending the work on the remaining types requires some routine, but tedious, additional details which we omit for readability.  

Each of Conditions \ref{E:A}-\ref{E:Z2} states that some random variable lies within some interval centered at its trajectory, i.e.\ it is equivalent to a statement of the form
\[
X(i) \in [x_1(t), x_2(t)],
\]
where $X=X(i)$ is our random variable and $x_1, x_2$ are deterministic functions of $t$ (possibly depending on $n$ as well). We use the following strategy to bound the probability that $X$ leaves the interval. First we define a pair of auxiliary random variables
\[
    X^+(i) := 
\begin{cases} 
X(i) -  x_2(t) & \text{ if $\mc{E}_{i-1}$ holds},\\
X^+(i-1) & \text{ otherwise},
\end{cases}
\]
and 
\[
   X^-(i) := 
\begin{cases} 
X(i) -  x_1(t) & \text{ if $\mc{E}_{i-1}$ holds},\\
X^-(i-1) & \text{ otherwise}.
\end{cases}
\]
Note that if $\mc{E}_{i-1}$ holds but $\mc{E}_i$ fails due to $X$ leaving its interval, then we have either $X^+(i)>0$ or $X^-(i)<0$. To bound the probability of that event we show that $X^+$ is a supermartingale, and $X^-$ is a submartingale (showing the latter is typically very similar to the former and we will often show less work here). The bound on the failure probability then follows from Freedman's inequality. 

We now proceed to apply the strategy described above to the variables of type $A$. We let
\[
A^{\pm}_{u'u'',k'}=A^{\pm}_{u'u'',k'}(i):=
\begin{cases} 
A_{u'u'',k'} -  n^2(a(t) \pm g_{ab}(t)) & \text{ if $\mc{E}_{i-1}$ holds},\\[3pt]
A^{\pm}_{u'u'',k'}(i-1) & \text{ otherwise}.
\end{cases}
\]

To check that $A^+_{u'u'',k'}$ is a supermartingale, 
we must show that $\Mean[\Delta \A^+_{u'u'',k'}|\mathcal F_i] \le 0$ where we define $\Delta \A^+_{u'u'',k'}:= \A^+_{u'u'',k'}(i+1) - \A^+_{u'u'',k'}(i)$. We first deal with a trivial case. If at step $i$ we have that $\mc{E}_i$ fails, then by definition we have $\Delta \A^+_{u'u'',k'}=0$ and we are done. Henceforth assume that $\mc{E}_i$ holds. 

We estimate the one-step change in $A_{u'u'',k'}$. This variable never increases, and each pair $(u, k) \in \A_{u'u'',k'}$ can be lost in one of the following ways:
\begin{enumerate}[label= $\bullet$]
    \item one of the vertices $u$, $u'$, $u''$ can get hit by $k$,
    \item one of the edges $uu'$, $uu''$ can have $k$ forbidden due to a potential alternating 4-cycle, or
    \item one of the edges $uu'$, $uu''$ can get colored.
\end{enumerate}
Thus, for each pair $(u, k) \in \A_{u'u'',k'}(i)$, the probability that $(u, k) \notin \A_{u'u'',k'}(i+1)$ is 
\[
   \P\sbrac{ \bigcup_{z \in \{u,u',u''\}} \mc{L}_{z,k} \;\;\; \cup \bigcup_{e \in \{uu', uu''\}} \rbrac{ \mc{N}_{e,k} \cup \mc{M}_{e, \bullet}}}
\]
and so
\[
\Mean[\Delta \A_{u'u'',k'}|\mathcal F_i] = -\sum_{(u, k)\in \A_{u'u'',k'}}  \P\sbrac{ \bigcup_{z \in \{u,u',u''\}} \mc{L}_{z,k} \;\;\; \cup \bigcup_{e \in \{uu', uu''\}} \rbrac{ \mc{N}_{e,k} \cup \mc{M}_{e, \bullet}}}.
\]
Now we will approximate the above probability by using the union bound with an error term as follows. Let $E_1,\dots,E_k$ be the set of events. Then,
\begin{equation}\label{prob:inequality}
\sum_{i=1}^k \P(E_i) - \sum_{1\le i<j\le k} \P(E_i\cap E_j) \le \P\sbrac{\bigcup_{i=1}^k E_i} \le \sum_{i=1}^k \P(E_i).
\end{equation}
This together with Claim~\ref{claim:intersections} and the assumption that the good event $\mc{E}_i$ holds implies that
\begin{align*}
   &\P\sbrac{ \bigcup_{z \in \{u,u',u''\}} \mc{L}_{z,k} \;\;\; \cup \bigcup_{e \in \{uu', uu''\}} \rbrac{ \mc{N}_{e,k} \cup \mc{M}_{e, \bullet}}}\\
   &\qquad\qquad\qquad =\sum_{z \in \{u,u',u''\}} \P\rbrac{  \mc{L}_{z,k}} + \sum_{e \in \{uu', uu''\}} \sbrac{\P\rbrac{ \mc{N}_{e,k}} + \P\rbrac{ \mc{M}_{e, \bullet}}} + O(n^{-3 + 12\d}).
\end{align*}
Consequently,
\begin{align*}
\Mean[\Delta \A_{u'u'',k'}|\mathcal F_i] &= -\sum_{(u, k)\in \A_{u'u'',k'}}  \cbrac{ \sum_{z = \{u,u',u''\}} \P\rbrac{  \mc{L}_{z,k}} + \sum_{e = \{uu', uu''\}} \sbrac{\P\rbrac{ \mc{N}_{e,k}} + \P\rbrac{ \mc{M}_{e, \bullet}}} + O(n^{-3 + 12\d})}.  \nn
\end{align*}
We will again use the assumption that $\mc{E}_i$ holds to give deterministic upper and lower bounds on $\Mean[\Delta \A_{u'u'',k'}|\mathcal F_i]$. Due Claim~\ref{claim:LMMN} we have 
\begin{align}
   \Mean[\Delta \A_{u'u'',k'}|\mathcal F_i] & \le \begin{multlined}[t] -n^2(a - g_{ab})\left\{ 3n^{-2} \sbrac{ \frac{5d}{6qc}  - \k \rbrac{p^{-8}g_{def}  + p^{-6} g_c}} \right. \nn\\
   \left. + 2 n^{-2}  \sbrac{\frac{3az_2}{qc}  + \frac{y}{q} - \k \rbrac{p^{-3} g_2 + p^{-5}g_{ab} + p^{-5} g_c}}   + O\rbrac{n^{-3 + 12\d} + n^{-5/2+4\d} }\right\}  \end{multlined}\nn\\
    & \le  -(a - g_{ab})\left\{  \frac{5d}{2qc}+  \frac{6az_2}{qc}  + \frac{2y}{q} - 10 \k \rbrac{p^{-3} g_2 + p^{-5}g_{ab}  + p^{-6} g_c}  \right\} + O(n^{-1/2 + 4\d}), \nn
\end{align}
where on the last line we used the fact that $p^{-3}g_{def} = g_{ab}$ and assumed that $\delta$ is small to simplify the big-O term. Now we will take the above expression and separate the ``main terms" from the ``first-order error term" (the terms involving error functions $g_{ab}$, etc.) and the ``lesser-order error terms" (in the big-O). We will be precise for the main terms and generous for error terms. 
By Claim~\ref{claim:helpersloppy} we get
\begin{align*}
    g_{ab}\left(\frac{5d}{2qc}+  \frac{6az_2}{qc}  + \frac{2y}{q}\right)&\le g_{ab}\left(125 p^{-1} + 60 + 20 p^{-1}\right)
    \le 205 g_{ab}p^{-1} \le \kappa g_{ab} p^{-1}.
\end{align*}
Recalling that $\k$ is large, $a(t) \le  p^5$ (see \eqref{eqn:atrajdef}), and $g_{ab}(t)=o(a(t))$, we obtain 
\begin{align*}
   \Mean[\Delta \A_{u'u'',k'}|\mathcal F_i] & \le - \frac{5ad}{2qc}-  \frac{6a^2z_2}{qc}  - \frac{2ay}{q} + 25 \k \rbrac{p^{2} g_2 + p^{-1}g_{ab} + p^{-1} g_c}   + O(n^{-1/2 + 4\d}).
\end{align*}
Similarly, we have
\begin{align}
   \Mean[\Delta \A_{u'u'',k'}|\mathcal F_i] & \ge  - \frac{5ad}{2qc}-  \frac{6a^2z_2}{qc}  - \frac{2ay}{q} - 25 \k \rbrac{p^{2} g_2 + p^{-1}g_{ab} + p^{-1} g_c}   + O(n^{-1/2 + 4\d}).\label{eqn:DeltaAlower}
\end{align}

We must also estimate the one-step change in $n^2(a+g_{ab})$, i.e.\ the deterministic part of $A^{\pm}_{u'u'',k'}.$ We use Taylor's theorem with the Lagrange form of the remainder: for a function $h: \mathbb R \rightarrow \mathbb R$ twice differentiable on $(x_0,x)$ and $h'$ continuous on $[x_0,x]$, we have
$$
h(x) - h(x_0) = h'(x_0)(x-x_0) + h''(x^*)(x-x_0)^2/2
$$
for some $x^*\in [x_0,x]$. In our case, $x_0=i/n^2 = t, x= (i+1)/n^2 = t + n^{-2}$. Thus for some $t^* \in [t, t+n^{-2}]$ we have
\begin{equation}\label{eqn:DeltaTrajA}
  \Delta n^2(a+g_{ab}) =  a'(t) + g'_{ab}(t) + \frac{a''(t^*) + g''_{ab}(t^*)}{2n^2} = a'(t) + g'_{ab}(t) + O(n^{-2}),  
\end{equation}
where the last expression follows from Propositions \ref{obs:CrudeDerivTraj} and \ref{obs:CrudeDerivErr}.

Putting \eqref{eqn:DeltaAlower} and \eqref{eqn:DeltaTrajA} together we have
\begin{align*}
\Mean[\Delta \A_{u'u'',k'}^+|\mathcal F_i] 
    & \le  - \frac{5ad}{2qc}-  \frac{6a^2z_2}{qc}  - \frac{2ay}{q} -a' -g_{ab}' + 25 \k \rbrac{p^{2} g_2 + p^{-1}g_{ab} + p^{-1} g_c}   +O(n^{-1/2 + 4\d})\\
    & = -g_{ab}' + 25 \k \rbrac{p^{2} g_2 + p^{-1}g_{ab}+ p^{-1} g_c}   + O(n^{-1/2 + 4\d})\\
    & \le -5 \k \rbrac{p^{2} g_2 + p^{-1}g_{ab} + p^{-1} g_c}   + O(n^{-1/2 + 4\d})\\
    & \le -\Omega\rbrac{n^{-\omega}},
\end{align*}
where the second line follows from \eqref{eqn:abdiffeq} which says $a' = - \frac{5ad}{2qc}-  \frac{6a^2z_2}{qc}  - \frac{2ay}{q}$,
the third line follows from~\eqref{eq:err-sup-ab}, and the final line follows from our choice of the error functions.
Thus $\A_{u'u'',k'}^+$ is a supermartingale. The reader can check that $\A_{u'u'',k'}^-$ is a submartingale using an entirely ``symmetric" calculation (i.e.\ we repeat the above calculation with the directions of inequalities reversed and the signs of the error terms reversed) using \eqref{eqn:DeltaAlower}.

We will apply Freedman's inequality from Lemma~\ref{lem:Freedman}. Our supermartingale will be $\A_{u'u'',k'}^+$. First we determine a suitable value for $D$. Note that at each step $i$, the number of edges that have a color forbidden (when it was available at step $i-1$) is $O(n)$. Also, any edge has $O(1)$ colors forbidden at each step. Thus, the number of pairs $(e, k)$ such that $k$ was available at $e$ at step $i-1$ but forbidden at step $i$ is $O(n)$. But the only way for a pair $(k, k')$ that is available  at a triple $(u, u', u'')$ at step $i-1$ to become forbidden at step $i$ is to forbid one of the colors $k, k'$ at one of the edges in $uu'u''$. Thus, we have $\Delta \A_{u'u'',k'} = O(n)$.

Meanwhile we have by \eqref{eqn:DeltaTrajA} and Propositions \ref{obs:CrudeDerivTraj} and \ref{obs:CrudeDerivErr} that
\[
\Delta n^2(a+g_{ab}) = a' + g_{ab}' + O(n^{-2}) = O(1)
\]
and so 
\[
|\Delta \A_{u'u'',k'}^+| \le |\Delta \A_{u'u'',k'}|+ |\Delta n^2(a+g_{ab})| = O(n).
\]
Thus, using that $|\Delta \A_{u'u'',k'}^+| = O(n)$ in the good event we get
\[
   \Var[\Delta \A_{u'u'',k'}^+ | \mc{F}_{k}]
   \le \Mean[ (\Delta \A_{u'u'',k'}^+)^2 | \mc{F}_{k}] 
   = O(n) \cdot \Mean[ |\Delta \A_{u'u'',k'}^+| | \mc{F}_{k}].
\]
In order to bound~$\Mean[ |\Delta \A_{u'u'',k'}^+| | \mc{F}_{k}]$, first observe that
\[
\Mean[ \Delta \A_{u'u'',k'} | \mc{F}_{k}] = O(1) \qquad \textrm{and} \qquad- \Mean[ \Delta \A_{u'u'',k'} | \mc{F}_{k}] = O(1),
\] 
by~\eqref{eqn:DeltaAlower}, and hence
\[
\Mean[ |\Delta \A_{u'u'',k'}^+| | \mc{F}_{k}]
\le \Mean[ |\Delta \A_{u'u'',k'}| | \mc{F}_{k}] + \Mean[ |\Delta n^2(a+g_{ab}) | | \mc{F}_{k}] = O(1).
\]
Consequently, $\Var[\Delta \A_{u'u'',k'}^+ | \mc{F}_{k}] = O(n)$ and for all $i \le i_{max} < \frac 16 n^2$ we have
\[
V(i) = \sum_{0 \le k \le i} \Var[ \A_{u'u'',k'}^+ | \mc{F}_{k}] = O(n^3).
\]
In view of the above calculations we are going to apply Freedman's inequality with $b=O(n^3)$ and $D=O(n)$.

We still need to estimate the initial value $\A_{u'u'',k'}^+(0)$ of our supermartingale. Note that 
\begin{equation}\label{eqn:AInit}
  \A_{u'u'',k'}^+(0) = \A_{u'u'',k'}(0) - n^2(a(0) + g_{ab}(0)).  
\end{equation}
Recall that $\A_{u'u'',k'}(0)$ is the number of pairs $(u, k)$ such that $(k, k')$ is available at $(u, u', u'')$ at step $0$. The only requirement here is that $k' \in S_u$, and $k \notin S_u, S_{u'}, S_{u''}$. Observe that $\A_{u'u'',k'}(0)$ is a binomial random variable with 
$\A_{u'u'',k'}(0)\sim\bin((n-2)|\COL|, s(1-s)^3)$. Thus, the expected value of $\A_{u'u'',k'}(0)$ is 
\[
\Mean[\A_{u'u'',k'}(0)] = (n-2)|\COL|s(1-s)^3 = n^2 a(0) + O(n) \quad \textrm{and}\quad \Mean[\A_{u'u'',k'}(0)] = \Theta(n^2)
\]
so an easy application of the Chernoff bounds~\eqref{Chernoff_upper} and~\eqref{Chernoff_lower} tells us that a.a.s. $|\A_{u'u'',k'}(0) - n^2a(0)| \le n^{3/2}$ for all $u', u'', k'$. Returning to \eqref{eqn:AInit} we have
\[
\A_{u'u'',k'}^+(0) \le n^{3/2} - n^2 g_{ab}(0) = n^{3/2} - n^{2-\omega} \le -\frac12 n^{2-\omega}.
\]
Thus for our application of Freedman's inequality we get to use $\lambda  = \frac12 n^{2-\omega}$.
Freedman's inequality then gives us that the probability $\A_{u'u'',k'}^+$ becomes positive before step $i_{max}$ is at most
\[
\displaystyle \exp\left(-\frac{\lambda^2}{2(b+D\lambda) }\right) = \exp\cbrac{-\Omega\rbrac{\frac{\rbrac{n^{2-\omega}}^2}{n^3 + n \cdot n^{2-\omega} }}}= \exp\cbrac{-\Omega\rbrac{ n^{1-\omega}}}.
\]
Since there are $O(n^3)$ choices for $u', u'', k'$, we have by the union bound that the probability any such choice ever sees $\A_{u'u'',k'}^+$ become positive before step $i_{max}$ is at most 
\[
O(n^3) \cdot \exp\cbrac{-\Omega\rbrac{ n^{1-\omega}}} = o(1).
\]

Similarly, one can apply Freedman's inequality to the supermartingales $-\A_{u'u'',k'}^-$ to show that the probability any of them become positive before step $i_{max}$ is $o(1)$. Thus, a.a.s.\ the good event $\mc{E}_{i_{max}}$ does not fail due to Condition \ref{E:A}. 

Handling Condition \ref{E:B} is similar, since the type $B$ variables are similar to type $A$ (in particular they even have the same trajectory).  
To demonstrate the similarity, note that for $(u'', k') \in \B_{uu',k}(i)$, the probability that $(u'', k') \notin \B_{uu',k}(i+1)$ is 
\begin{align}
   &\P\sbrac{ \mc{L}_{u'',k} \cup \bigcup_{z=\{u',u''\}} \mc{L}_{z,k'}\;\;\; \cup \mc{N}_{uu'',k} \cup\mc{N}_{u'u'',k'}  \cup \bigcup_{e=\{uu'',u'u''\}} \mc{M}_{e, \bullet}}.\nn
\end{align}
And although the indices are different, there are exactly the same number of each of the events $\mc{L}_{z,k^*},$ $\mc{N}_{e,k^*},$ $\mc{M}_{e,\bullet}$, which will yield precisely the same estimates as $\A_{u'u'',k'}$. Thus, to avoid too much repetition we will not show the work for Condition \ref{E:B}.

\section{Variable \texorpdfstring{$C^{(1)}$}{}}\label{sec:C1C2}

In this section, we address \ref{E:C1}. 
Now define 
\[
{\CIpm}={\CIpm}(i):=
\begin{cases} 
\CI  - n ( c_1(t)\pm g_{c1} ) & \text{if $\mc{E}_{i-1}$ holds},\\[3pt]
{\CIpm}(i-1)  & \text{otherwise}.
\end{cases}
\]
We demonstrate that $\CIp$ is a supermartingale. To estimate the one-step change, note that we may lose $k' \in \CI(i)$ if $u',u''$ is hit by $k'$ or if $u'u''$ becomes part of an alternating $(u'u'',k')$-path. Thus, due to~\eqref{prob:inequality} and Claim~\ref{claim:intersections}, we get
\begin{align*}
\Mean[\Delta \CI|\mathcal F_i]
&= -\sum_{k'\in \CI} \P\sbrac{
\bigcup_{z \in \{u',u''\}} \rbrac{\mc{L}_{z,k'} \cup
\mc{N}_{u'u'',k'}}}\\
&= -\sum_{k'\in \CI} \sbrac{
\sum_{z \in \{u',u''\}} \P(\mc{L}_{z,k'}) +
\P(\mc{N}_{u'u'',k'}) + O(n^{-3+8\delta})}.
\end{align*}
Now, Claim~\ref{claim:LMMN} yields
\begin{align*}
\Mean[\Delta \CI|\mathcal F_i]
& \le \begin{multlined}[t]
    -n(c_1-g_{c_1})\cbrac{ 2n^{-2} \sbrac{ \frac{5d}{6qc}  - \k \rbrac{p^{-8}g_{def}  + p^{-6} g_c}} \right.\\ \left.
    \qquad\qquad\qquad\qquad+   n^{-2}  \sbrac{\frac{3az_2}{qc}   - \k \rbrac{p^{-3} g_2 + p^{-5}g_{ab}  + p^{-5} g_c}}} + O(n^{-2+8\delta})
\end{multlined}\\
& \le -n^{-1}(c_1-g_{c_1}) \left[\frac{5d}{3qc}+\frac{3az_2}{qc} - 2\k(p^{-3}g_2 + p^{-5}g_{ab} + p^{-6}g_c+p^{-8}g_{def})\right] + O(n^{-2+8\delta})\\
& \le n^{-1}\left(-\frac{5dc_1}{3qc}-\frac{3az_2c_1}{qc} + 20 \k (p^{-1}g_2 + p^{-3}g_{ab}+ p^{-4}g_c+p^{-6}g_{def}) \right) + O(n^{-2+8\delta}),
\end{align*}
where in the last line we use bounds from Claim~\ref{claim:helpersloppy} together with $g_1 = o(c_1)$ and $c_1 \le p^2$. The lower bound will follow by symmetric calculations. 

Now by Taylor's theorem we have $\Delta(n(c_1 + g_{c_1})) = n^{-1}(c_1' + g_{c_1}') + O(n^{-3})$. Therefore, in the good event by applying~\eqref{eqn:c1diffeq} and~\eqref{eq:err-sup-c1}, we obtain 
\begin{align*}
\begin{split}
\Mean[\Delta \CIp|\mathcal F_i] 
    &\le n^{-1}\left(-c_1'  -\frac{5dc_1}{3qc}-\frac{3az_2c_1}{qc} - g_{c_1}' \right.\\
    &\left.\qquad\qquad+ 20 \k (p^{-1}g_2 + p^{-3}g_{ab} + p^{-4}g_c+p^{-6}g_{def})\vphantom{\frac{3az_2c_1}{qc}}\right) + O(n^{-2+8\delta}) 
\end{split}\\
& = n^{-1}\left(- g_{c_1}'+ 20 \k (p^{-1}g_2 + p^{-3}g_{ab} + p^{-4}g_c+p^{-6}g_{def})\right)+ O(n^{-2+8\delta})\\
& \le n^{-1}\left(-10 \k (p^{-1}g_2 + p^{-3}g_{ab} + p^{-4}g_c+p^{-6}g_{def})\right)+ O(n^{-2+8\delta})\\
&\le -\Omega(n^{-1-\omega}).
\end{align*}

Now to apply Freedman's inequality, we estimate the maximum one-step change of $\CI$. Since the number of ways to forbid a color at an edge in one step is $O(1)$, we get that that $|\Delta \CI|=\Delta \CI = O(1)$, and by Propositions~\ref{obs:CrudeDerivTraj} and \ref{obs:CrudeDerivErr}, $\Delta n(c_1+g_{c_1}) = n^{-1}(c_1' + g_{c_1}') + O(n^{-3})$ yielding
\[
|\Delta n(c_1+g_{c_1})| \le n^{-1}(|c_1'| + |g_{c_1}'|) + O(n^{-3}) = O(n^{-1}|c_1'|) = O(n^{-1})
\]
and so 
\[
|\Delta \CIp| \le |\Delta \CI|+ |\Delta n(c_1+g_{c_1})| = O(1).
\]
Thus we let $D=O(1)$ in Freedman's inequality. Further, we have
\[
   \Var[ \Delta \CIp | \mc{F}_{k}] 
   \le \Mean[ (\Delta \CIp)^2 | \mc{F}_{k}] 
   = O(1) \cdot \Mean[ |\Delta \CIp| | \mc{F}_{k}]
   = O(n^{-1}).
\]
Therefore, $V(i) = O(n)$ for all $i \le i_{max}$ and so we take $b= O(n)$. In addition, Chernoff's bound allows us to take $\lambda = \frac{1}{2}n^{1-\omega}$, and so Freedman's inequality demonstrates that the probability that $\CIp$ becomes positive before step $i_{max}$ is at most 
$\exp\cbrac{-\Omega\rbrac{ n^{1-2\omega}}}$, which beats the union bound over all $O(n^3)$ choices for $u$, $u'$ and $u''$.

\section{Variable \texorpdfstring{$D$}{}}\label{sec:DEF}
In this section, we address \ref{E:D}. Since \ref{E:E}--\ref{E:F} are very similar and these variables share the same trajectory, we will omit their calculations (see our discussion of $B$ type variables in Section~\ref{sec:AB}). 
Define
\[
\D^{\pm}_{u,k}=\D^{\pm}_{u,k}(i):=
\begin{cases} 
\D_{u,k}  - n^3 ( d(t)\pm g_{def} ) & \text{if $\mc{E}_{i-1}$ holds},\\[3pt]
\D^{\pm}_{u,k}(i-1) & \text{otherwise}.
\end{cases}
\]
To bound the expected one-step change, note that we can lose $(u', u'', k') \in \D_{u,k}$ in several ways: one of the edges could become matched, $u'$ or $u''$ could become hit by $k$ or $k'$, or one of the edges could become part of an alternating path. Hence, using Claims~\ref{claim:intersections}, \ref{claim:LMMN} and~\ref{claim:helpersloppy} together with $g_{def}=o(d)$, and $d\le p^7$, yield
\begin{align*}
\Mean[\Delta \D_{u,k}&|\mathcal F_i]
= -\sum_{(u',u'',k')\in \D_{u,k}} \P\left(\bigcup_{e \in \{uu',uu'', u'u''\}} \mc{M}_{e,\bullet} \cup \bigcup_{z \in \{u',u''\}} ((\mc{L}_{z,k}) \cup (\mc{L}_{z,k'}))  \cup\bigcup_{e \in \{uu', uu''\}} \mc{N}_{e,k} \cup \mc{N}_{u'u'', k'}\right) \\
\begin{split}
    &=-\sum_{(u',u'',k')\in \D_{u,k}} \left(\sum_{e \in \{uu',uu'',u'u''\}} \P(\mc{M}_{e,\bullet}) + \sum_{z \in \{u',u''\}} (\P(\mc{L}_{z,k})+ \P(\mc{L}_{z,k'})) \right.\\
    &\left.\qquad\qquad\qquad+\sum_{e \in \{uu',uu''\}} \P(\mc{N}_{e,k}) + \P(\mc{N}_{u'u'', k'})+O(n^{-3+12\delta})\right) 
\end{split}\\    
&\le \begin{multlined}[t]
    -n^{3}(d-g_{def})\left\{ 4n^{-2} \sbrac{ \frac{5d}{6qc}  - \k \rbrac{p^{-8}g_{def}  + p^{-6} g_c}} \right.\\ \left.
    \qquad\qquad+\ \  3 n^{-2}  \sbrac{\frac{3az_2}{qc}  + \frac{y}{q}  - \k \rbrac{p^{-3} g_2 + p^{-5}g_{ab} + p^{-5} g_c}}+ O(n^{-5/2+4\delta})\right\} + O(n^{12\delta})
\end{multlined}\\
&\le -(d-g_{def}) \sbrac{\frac{20d}{6qc}+\frac{9az_2}{qc}+\frac{3y}{q} - 7 \kappa \rbrac{p^{-3}g_2 + p^{-8}g_{def} + p^{-5}g_{ab}+p^{-6}g_c}}n + O(n^{1/2+4\delta})\\
&\le \sbrac{-\frac{20d^2}{6qc}-\frac{9az_2d}{qc}-\frac{3yd}{q} + 20 \kappa \rbrac{p^{4}g_2 + p^{-1}g_{def} + p^{2}g_{ab}+pg_c}}n + O(n^{1/2+4\delta}).
\end{align*}

On the other hand, by Taylor's theorem we have $\Delta(n^3 (d + g_{def})) = (d' + g_{def}')n + O(n^{-1})$. Therefore, in the good event due to \eqref{eqn:defdiffeq} and \eqref{eq:err-sup-def}, we get
\begin{align*}
\Mean[\Delta \D^{+}_{u,k}|\mathcal F_i] 
&\le \sbrac{-d'-\frac{20d^2}{6qc}-\frac{9az_2d}{qc}-\frac{3yd}{q} - g_{def}'+ 20 \kappa \rbrac{p^{4}g_2 + p^{-1}g_{def} + p^{2}g_{ab}+pg_c}}n + O(n^{1/2+4\delta})\\
&= \sbrac{- g_{def}'+ 20 \kappa \rbrac{p^{4}g_2 + p^{-1}g_{def} + p^{2}g_{ab}+pg_c}}n + O(n^{1/2+4\delta})\\
&\le\sbrac{- 10 \kappa \rbrac{p^{4}g_2 + p^{-1}g_{def} + p^{2}g_{ab}+pg_c}}n + O(n^{1/2+4\delta})\\
&\le - \Omega(n^{1-\omega}).
\end{align*}

As before, we apply Freedman's inequality to $\D^{+}_{u,k}$ by first estimating the maximum one-step change of $D_{u,k}$. As discussed above, the maximum one-step change is $O(n^2)$ by having at most $O(n)$ edges $e$ forbid the $O(n)$ pairs $(e,k')$. In addition, Propositions~\ref{obs:CrudeDerivTraj} and \ref{obs:CrudeDerivErr} imply
\[
|\Delta n^3(d+g_{def})| \le n(|d'| + |g_{def}'|) + O(n^{-1}) = O(n|d'|) = O(n)
\]
and so 
\[
|\Delta \D_{u,k}^+| \le |\Delta \D_{u,k}|+ |\Delta n^3(d+g_{def})| = O(n^2),
\]
so we take $D=O(n^2)$. Furthermore,
\[
   \Var[ \Delta \D_{u,k}^+ | \mc{F}_{k}] 
   \le \Mean[ (\Delta \D_{u,k}^+)^2 | \mc{F}_{k}] 
   = O(n^2) \cdot \Mean[ |\Delta \D_{u,k}^+| | \mc{F}_{k}]
   = O(n^3)
\]
and so for all $i \le i_{max} < \frac 16 n^2$ we have $V(i) = O(n^5)$.

Therefore, take $b= O(n^5)$ and $\lambda = \frac{1}{2}n^{3-\omega}$ to get by Freedman's inequality and the union bound a failure probability of $O(n^2)\cdot\exp(-\Omega(n^{1-2\omega})) = o(1)$.  

\section{Variable \texorpdfstring{$Z_0$}{}}\label{sec:Z0}

In this section, we address \ref{E:Z0} by considering $Z_0$. Extending the work on the remaining variables $Z_1$ and $Z_2$ requires some similar calculations (some of which involve the bounds in \ref{E:crude}), which we omit for readability.  

Define
$$
{Z_{uv, k, 0,0,0}^\pm}={Z_{uv, k, 0,0,0}^{\pm}}(i):=
\begin{cases} 
Z_{uv, k, 0,0,0}  - n^3 ( z_0(t)\pm g_{0} ) & \text{if $\mc{E}_{i-1}$ holds},\\[3pt]
Z_{uv,k,0,0,0}^\pm(i-1) & \text{otherwise}.
\end{cases}
$$
Notice that the expected one-step change can never increase and we may lose a $(x,y,k')\in Z_{uv, k, 0,0,0}$ in several ways: the vertex $x,y$ is hit with the color $k$ or $x,y,u,v$ is hit with the color $k'$; or the edge $xy, ux, vy$ is colored; or $ux, vy, xy$ becomes part of an alternating path. Thus, by Claims~\ref{claim:intersections}, \ref{claim:LMMN} and~\ref{claim:helpersloppy} together with bounds $g_0=o(z_0)$, and $z_0\le p^9$, we get
\begin{align*}
\begin{split} \Mean[\Delta Z_{uv,k,0,0,0}|&\mathcal F_i] = 
    -\sum_{(x,y, k')\in Z_{uv,k,0,0,0}} \P\left( \bigcup_{z \in\{ x,y\}} \mc{L}_{z,k} \cup\bigcup_{z \in\{ x,y,u,v\}} \mc{L}_{z,k'} \right.\\
    &\qquad\qquad\qquad\qquad\qquad\qquad\left.\cup \bigcup_{e \in \{xy, ux, vy\}} \mc{M}_{e,\bullet} \cup \bigcup_{e\in\{ux, vy\}} \mc{N}_{e,k'} \cup \mc{N}_{xy,k} 
    \vphantom{\bigcup_{(x,y, k')}}\right)  
\end{split}\\
\begin{split} 
    &=-\sum_{(x,y, k')\in Z_{uv,k,0,0,0}} \left( \sum_{z \in \{x,y\}} \P(\mc{L}_{z,k}) +\sum_{z \in \{x,y,u,v\}} \P(\mc{L}_{z,k'}) \right.\\
    &\qquad\left.+ \sum_{e \in \{xy, ux, vy\}} \P(\mc{M}_{e,\bullet}) + \sum_{e\in\{ux, vy\}} \P(\mc{N}_{e,k'}) + \P(\mc{N}_{xy,k}) + O(n^{-3+12\delta})
    \vphantom{\bigcup_{(x,y, k')}}\right)  
\end{split}\\
\begin{split}
    &\le -n^3(z_0 - g_0) \left\{6n^{-2} \sbrac{ \frac{5d}{6qc}  - \k \rbrac{p^{-8}g_{def}  + p^{-6} g_c}} \right.\\
    &\qquad\left.+  3 n^{-2}  \sbrac{\frac{3az_2}{qc}  + \frac{y}{q} - \k \rbrac{p^{-3} g_2 + p^{-5}g_{ab}  + p^{-5} g_c} + O(n^{-5/2+4\delta}) \vphantom{}}\right\}
    + O\rbrac{n^{12\delta}} 
\end{split}\\
\begin{split}
    & = -n(z_0-g_0) \left[\frac{5d}{qc} + \frac{9az_2}{qc} + \frac{3y}{q} \right.\\
    &\qquad\left. - 9 \kappa \rbrac{p^{-3}g_2 + p^{-8}g_{def} + p^{-5}g_{ab}+p^{-6}g_c\vphantom{\frac{10d}{3qc}} }
     \right] + O\rbrac{n^{1/2+4\delta}} 
\end{split}\\
&\le n\sbrac{-\frac{5dz_0}{qc} - \frac{9az_2z_0}{qc} - \frac{3yz_0}{q} + 20 \k \rbrac{p^{6}g_2 + pg_{def} + p^{4}g_{ab}+p^{3}g_c} \vphantom{\frac{10dz_0}{3qc}}} + O\rbrac{n^{1/2+4\delta}}. 
\end{align*}
By Taylor's theorem we have $\Delta(n^3 (z_0 + g_{0})) = (z_0' + g_{0}')n + O(n^{-1})$. Therefore in the good event by~\eqref{eqn:z0diffeq} and~\eqref{eq:err-sup-0}, we obtain
\begin{align*}
\begin{split}
\Mean[\Delta Z^{+}_{uv,k,0,0,0}|\mathcal F_i] 
&\le \left[-z_0'-\frac{5dz_0}{qc} - \frac{9az_2z_0}{qc} - \frac{3yz_0}{q} - g_0' \right.\\ 
&\qquad\left.+ 20 \k \rbrac{p^{6}g_2 + pg_{def} + p^{4}g_{ab}+p^{3}g_c\vphantom{\frac{10dz_0}{3qc}} } \right]n + O\rbrac{n^{1/2+4\delta}} \\
&= \left[- g_0' + 20 \k \rbrac{p^{6}g_2 + pg_{def} + p^{4}g_{ab}+p^{3}g_c\vphantom{\frac{10dz_0}{3qc}} } \right]n + O\rbrac{n^{1/2+4\delta}}\\
&\le \left[-10 \k \rbrac{p^{6}g_2 + pg_{def} + p^{4}g_{ab}+p^{3}g_c\vphantom{\frac{10dz_0}{3qc}} } \right]n + O\rbrac{n^{1/2+4\delta}}\\
&\le -\Omega(n^{1-\omega}).
\end{split}
\end{align*}

Consider $Z_{uv, k, 0, 0, 0}$ for some fixed edge $uv$ and color $k$. The one-step change in this random variable never has any positive contributions, and its negative contributions can come in several ways. Suppose $(x, y, k') \in Z_{uv, k, 0, 0, 0}(i)$. Then we could have $(x, y, k') \notin Z_{uv, k, 0, 0, 0}(i+1)$ for any of the following (exhaustive) list of reasons:
\begin{enumerate}[label=(\roman*)]
    \item\label{z:reason1} one of the edges $ux, xy, yv$ gets colored,
    \item\label{z:reason2} one of the vertices $x, y$ gets hit by one of the colors $k, k'$,
    \item\label{z:reason3} one of the vertices $u, v$ gets hit by $k'$,
    \item\label{z:reason4} $k$ is forbidden at $xy$ through an alternating 4-cycle,
    \item\label{z:reason5} $k'$ is forbidden at $ux$ or $yv$ through an alternating 4-cycle.
\end{enumerate}
Consider the triples $(x, y, k')$ that are removed from $Z_{uv, k, 0, 0, 0}$ due to~\ref{z:reason1}. Two of the vertices in $\{u,x,y,v\}$ must be in the triangle that gets colored in this step, and so the number of triples $(x, y, k')$ is at most $O(n^2)$. Reason~\ref{z:reason2} is similarly $O(n^2)$. Reason~\ref{z:reason3} is $O(n^2)$ since $k'$ must be one of the colors in the triangle getting colored. Now in~\ref{z:reason4}, we observe that for a fixed color $k$, in a single step $k$ is forbidden at $O(n)$ many edges due to potential 4-cycles. Since $xy$ would have to be one of those edges, we get $O(n^2)$. Now for~\ref{z:reason5}, observe that for a fixed color $k'$, there are at most $O(1)$ edges $ux$ adjacent to $u$ such that $k'$ is forbidden at $ux$ through a potential 4-cycle. Thus, we obtain again $O(n^2)$.  

We now apply Freedman's inequality. Note that
\[
|\Delta n^3(z_0+g_{0})| \le n(|z_0'| + |g_{0}'|) + O(n^{-1}) = O(n|z_0'|) = O(n^{1+8\d})
\]
and so 
\[
|\Delta Z_{uv,k,0,0,0}^+| \le |\Delta Z_{uv,k,0,0,0}|+|\Delta n^3(z_0+g_{0})| = O(n^2).
\]
Therefore, we let $D = O(n^2)$.
In addition,
\[
   \Var[ \Delta Z_{uv,k, 0,0,0}^+ | \mc{F}_{k}] 
   \le \Mean[ (\Delta Z_{uv,k,0,0,0}^+)^2 | \mc{F}_{k}] 
   \le O(n^2) \cdot \Mean[ |\Delta Z_{uv,k,0,0,0}^+| | \mc{F}_{k}]
   = O(n^3)
\]
implying that $V(i) = O(n^5)$.
Therefore, take $b= O(n^5)$. Using Chernoff's bound to estimate $Z_{uv,k,0,0,0}^+(0)$ allows us to set $\lambda = \frac{1}{2}n^{3-\omega}$. Thus Freedman's inequality gives us an exponentially small failure probability. 

\section{Bounds on \texorpdfstring{$\Xi$}{}, \texorpdfstring{$\Phi$}{},  \texorpdfstring{$\Psi$}{}, \texorpdfstring{$\Lambda$}{}}\label{sec:crude}

In this section we bound the probability that the good event $\mc{E}_{i_{max}}$ fails due to Condition \ref{E:crude}. The variables we are bounding here are all similar, so we will only show the details for $\Xi_{u, v, k}$. First we define versions of these variables that are ``frozen" outside the good event $\mc{E}_{i-1}$:
\begin{equation*}
    \Xi_{u, v, k}^*(i) := \begin{cases} 
\Xi_{u, v, k}(i) &  
\mbox{ if $\mc{E}_{i-1}$ holds},
\vspace{2ex}\\
\Xi_{u, v, k}^*(i-1) & \mbox{ otherwise}.
\end{cases}
\end{equation*}

Note 
%$\Xi_{u, v, k}(i)$ (and therefore $\Xi_{u, v, k}^*(i)$) is nondecreasing in $i$. Also 
that we have $\Delta \Xi_{u, v, k}(i) =O(1)$ and therefore $\Delta \Xi_{u, v, k}^*(i)=O(1)$. We bound the probability that $\Delta \Xi_{u, v, k}(i) \neq 0$ as follows. First we bound the number of ``predecessors," (see figures below) i.e. triples $(x, y, k')$ which are not in $\Xi_{u, v, k}(i-1)$ but which could become an element of $\Xi_{u, v, k}(i)$. On the first row in Figure~\ref{fig10} below we have ``single-edge predecessors" that only need one edge colored in order to become part of $\Xi_{u, v, k}(i)$. On the second row we see ``double-edge predecessors" which need two edges colored simultaneously (of course, for two edges to get colored in one step they would need to be adjacent). 

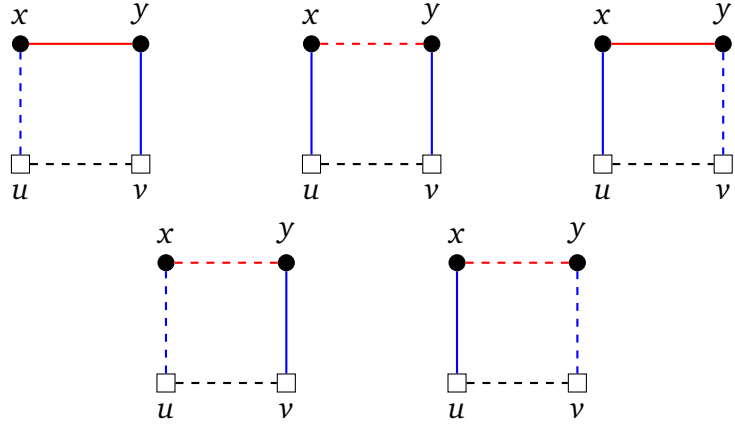
\begin{figure}[h!]
\begin{minipage}{\textwidth}
\begin{center}
\begin{tikzpicture}[scale=.8]
	\node (u) at (0,0) [fixed,label=below:$u$] {};
	\node (v) at (2,0) [fixed,label=below:$v$] {};
	\node (y) at (2,2) [vert,label=above:$y$] {};
	\node (x) at (0,2) [vert,label=above:$x$] {};
	\draw [uncolored] (u) -- (v);
	\draw [blue] (y) -- (v);
	\draw [blue?] (u) -- (x);
	\draw [red] (y) -- (x);
\end{tikzpicture}
\qquad \qquad 
\begin{tikzpicture}[scale=.8]
	\node (u) at (0,0) [fixed,label=below:$u$] {};
	\node (v) at (2,0) [fixed,label=below:$v$] {};
	\node (y) at (2,2) [vert,label=above:$y$] {};
	\node (x) at (0,2) [vert,label=above:$x$] {};
	\draw [uncolored] (u) -- (v);
	\draw [blue] (y) -- (v);
	\draw [blue] (u) -- (x);
	\draw [red?] (y) -- (x);
\end{tikzpicture}
\qquad \qquad 
\begin{tikzpicture}[scale=.8]
	\node (u) at (0,0) [fixed,label=below:$u$] {};
	\node (v) at (2,0) [fixed,label=below:$v$] {};
	\node (y) at (2,2) [vert,label=above:$y$] {};
	\node (x) at (0,2) [vert,label=above:$x$] {};
	\draw [uncolored] (u) -- (v);
	\draw [blue?] (y) -- (v);
	\draw [blue] (u) -- (x);
	\draw [red] (y) -- (x);
\end{tikzpicture}
\end{center}
\end{minipage}

\begin{minipage}{\textwidth}
\begin{center}
\begin{tikzpicture}[scale=.8]
	\node (u) at (0,0) [fixed,label=below:$u$] {};
	\node (v) at (2,0) [fixed,label=below:$v$] {};
	\node (y) at (2,2) [vert,label=above:$y$] {};
	\node (x) at (0,2) [vert,label=above:$x$] {};
	\draw [uncolored] (u) -- (v);
	\draw [blue] (y) -- (v);
	\draw [blue?] (u) -- (x);
	\draw [red?] (y) -- (x);
\end{tikzpicture}
\qquad \qquad 
\begin{tikzpicture}[scale=.8]
	\node (u) at (0,0) [fixed,label=below:$u$] {};
	\node (v) at (2,0) [fixed,label=below:$v$] {};
	\node (y) at (2,2) [vert,label=above:$y$] {};
	\node (x) at (0,2) [vert,label=above:$x$] {};
	\draw [uncolored] (u) -- (v);
	\draw [blue?] (y) -- (v);
	\draw [blue] (u) -- (x);
	\draw [red?] (y) -- (x);
\end{tikzpicture}
\end{center}
\end{minipage}
\caption{Depictions of ``double-edge predecessors''  (on the first row) and ``single-edge predecessors'' (on the second row) of $\Xi_{u, v, k}(i)$.}
\label{fig10}
\end{figure}

Note that the number of single-edge predecessors is $O(n)$ since for each fixed $k'$ there is a constant number of choices for $x, y$. For a single-edge predecessor triple $(x, y, k')$ to become part of $\Xi_{u, v, k}(i)$, a particular edge needs to get a particular color, which has probability $O\rbrac{n^{-3+3\d}}$ in the good event. The number of double-edge predecessors is $O(n^2)$, and for one of them to become part of $\Xi_{u, v, k}(i)$ we need to color a particular triangle using a particular pair of colors, which has probability $O\rbrac{n^{-5+8\d}}$ in the good event. Thus, in the good event we have
\[
\P\rbrac{\Delta \Xi_{u, v, k}(i) \neq 0} = O(n) \cdot O\rbrac{n^{-3+3\d}} + O(n^2) \cdot O\rbrac{n^{-5+8\d}} = O\rbrac{n^{-2+3\d}}.
\]
Of course this implies $\P\rbrac{\Delta \Xi_{u, v, k}^*(i) \neq 0} = O\rbrac{n^{-2+3\d}}$ as well. Thus, the final value $\Xi_{u, v, k}^*(i_{max})$ is stochastically dominated by $X \sim K \bin(i_{max}, Kn^{-2+3\d})$ for some constant $K$. An easy application of Chernoff shows that 
\[
\P\rbrac{X > n^{4\d}} \le \exp\cbrac{-\Omega(n^{4\d})}.
\]
Since there are only a polynomial number of variables $\Xi_{u, v, k}$, the union bound shows that a.a.s. none of them exceed $n^{4\d}$.

\section{Finishing the coloring}\label{sec:finishing}

In this section we describe Phase 2 of our coloring procedure. We assume that Phase 1 has terminated successfully (i.e.\ the event $\mc{E}_{i_{max}}$ holds). In Phase 2 we will assign  to each uncolored edge a uniform random color from the $\eps n/2$ colors in $\overline{\COL} \setminus \COL$. We will use the Lov\'asz Local Lemma (Lemma~\ref{lem:LLL}) to show there is a positive probability that none of the following ``bad" events occur in Phase 2. Here when we say ``uncolored edges" we mean edges that were not colored in Phase 1. Define the following events:
\begin{enumerate}[label=$\bullet$]
    \item For two adjacent uncolored edges $e_1, e_2$ let $B_1(e_1, e_2)$ be the event that both edges get the same color.
    \item For any 4-cycle of uncolored edges $e_1, e_2, e_3, e_4$, let $B_2(e_1, e_2, e_3, e_4)$ be the event that this 4-cycle becomes alternating (i.e.\ $e_1$ gets the same color as $e_3$, and $e_2$ gets the same color as $e_4$).
    \item For any 4-cycle of edges $e_1, e_2, e_3, e_4$ such that $e_1$ and $e_3$ are uncolored and $e_2$ and $e_4$ were given the same color in Phase 1, let $B_3(e_1,e_3)$ be the event that this 4-cycle becomes alternating (i.e.\ $e_1$ gets the same color as $e_3$).
\end{enumerate}
Let $\mc{B}$ be the family of all bad events of types $B_1, B_2, B_3$ described above. Note that if none of the events in $\mc{B}$ happens, then Phase 2 gives us a $(4, 5)$-coloring. 

Toward describing our dependency graph we claim the following:
\begin{claim}
 Fix any event $B \in \mc{B}$ (of any type $B_1$, $B_2$ or $B_3$). Among the other events in $\mc{B}$, $B$ is mutually independent with all but at most 
\begin{enumerate}[label=$\bullet$]
    \item  $O(n^{1-\d})$ events of type $B_1$,
    \item  $O(n^{2-2\d})$ events of type $B_2$, and
    \item $O(n^{1-\d})$ events of type $B_3$.
\end{enumerate}
\end{claim}

\begin{proof}
Every event in $\mc{B}$ involves some set of uncolored edges and the colors they get in Phase 2. Any such event $B$ is mutually independent of the set of all events $B'$ that do not involve any of the same edges as $B$. So, for each type (i.e.\ type $B_1, B_2,$ or $B_3$) we bound the number of $B'$ of that type sharing an edge with $B$. 

We show that any fixed uncolored edge $e_1$ is in $O(n^{1-\d})$ events of the form $B_1(e_1, e_2)$. Indeed, this will follow from bounding the number of uncolored edges at a vertex. Bohman, Frieze and Lubetzky~\cite{BFL15} proved that in the triangle removal process the degree of each vertex is a.a.s.\ $(1+o(1))np$ as long as we have, say, $p \ge n^{-1/3}$ (the power of $n$ could be any constant larger than $-1/2$). In our analysis we are requiring the stronger condition $p \ge n^{-\d}$ (which is the value of $p$ at step $i_{max}$ when we stopped the Phase 1 process). Thus, at the end of Phase 1 each vertex is incident with $O(n^{1-\d})$ uncolored edges. 

Next we show that any fixed uncolored edge $e_1$ is in $O(n^{2-\d})$ events of the form $B_2(e_1, e_2, e_3, e_4)$. But, given $e_1$ and our bound on degrees, there are $O(n^{1-\d})$ choices for $e_2$ and then $O(n^{1-\d})$ choices for $e_3$, which determines at most one choice for $e_4$ and we are done. 

Finally, we show that any fixed uncolored edge $e_1$ is in $O(n^{1-\d})$ events of the form $B_3(e_1, e_3)$. We know that at the end of Phase 1 we have
\[
\sum_{k \in \COL} Z_{e_1, k, 1, 0, 1} = O\rbrac{|\COL| \cdot n^{1-3\d}} =  O\rbrac{ n^{2-3\d}}.
\]
Each event $B_3(e_1, e_3)$ is counted in the above sum once for every color $k$ available at $e_3$. So we estimate the number of colors available at an edge. Say $u', u''$ are the endpoints of $e_3$. We know that 
\[
\sum_{u \in V} \CI = \Theta\rbrac{n \cdot n^{1-2\d}} = \Theta\rbrac{ n^{2-2\d}}.
\]
For each color $k'$ available at $e_3$, the sum above counts $k'$ once for every vertex $u$ such that $k' \in S_u$. An easy application of the Chernoff bound gives us that a.a.s. for every color $k'$ there are 
$(1+o(1)) ns = \Theta(n)$ vertices $u$ such that $k' \in S_u$.  Thus the number of colors available at $e_3$ is $\Theta\rbrac{ n^{1-2\d}}$. Thus, the number edges $e_3$ such that we have a bad event $B(e_1, e_3)$ is 
\[
O\rbrac{ \frac{n^{2-3\d}}{n^{1-2\d}}} = O\rbrac{ n^{1-\d}},
\]
as required.
\end{proof}

To apply the Local Lemma we must assign to each bad event $B \in \mc{B}$ a number $x_B \in [0, 1)$. To all the events of type $B_j$  we assign the number $x_j$ ($j=1,2,3$), where
\[
x_1:= \frac{10}{\eps n}, \qquad x_2:= \frac{10}{(\eps n)^2}, \qquad  x_3:= \frac{10}{\eps n}.
\]
We check the condition \eqref{eqn:LLLcond} of the Local Lemma. Since Phase 2 uses the set $\overline{\COL} \setminus \COL$ of $\eps n/2$ colors, the probability of any $B_1$ event is $2/(\eps n)$, which is smaller than
\[
x_1 (1-x_1)^{O(n^{1-\d})} (1-x_2)^{O(n^{2-2\d})} (1-x_3)^{O(n^{1-\d})} = (1+o(1)) x_1.
\]
The probability of any $B_2$ event is $4/(\eps n)^2$, which is smaller than
\[
x_2 (1-x_1)^{O(n^{1-\d})} (1-x_2)^{O(n^{2-2\d})} (1-x_3)^{O(n^{1-\d})} = (1+o(1)) x_2.
\]
The probability of any $B_3$ event is $2/(\eps n)$, which is smaller than
\[
x_3 (1-x_1)^{O(n^{1-\d})} (1-x_2)^{O(n^{2-2\d})} (1-x_3)^{O(n^{1-\d})} = (1+o(1)) x_3.
\]
Thus, the conditions of Lemma \ref{lem:LLL} are met and so with positive probability Phase 2 gives us a $(4, 5)$-coloring. This completes the proof of Theorem \ref{thm:main}.

\end{document}